\documentclass[11pt,a4paper]{amsart}

 \usepackage{amsfonts,amsmath,amscd,amssymb,amsbsy,amsthm,amstext,amsopn,amsxtra}
 \usepackage{fullpage,mathrsfs,subfigure}
 \usepackage{pstricks,pst-node,pst-text,pst-tree,multicol}
 \usepackage{stmaryrd}
 \usepackage{extarrows}
 \usepackage[all]{xy}

 \numberwithin{equation}{section}
\newtheorem{theorem}{Theorem}[section]
\newtheorem{proposition}[theorem]{Proposition}
\newtheorem{lemma}[theorem]{Lemma}

\newtheorem{corollary}[theorem]{Corollary}
\newtheorem{conjecture}[theorem]{Conjecture}

\theoremstyle{definition}
\newtheorem{definition}[theorem]{Definition}
\newtheorem{example}[theorem]{Example}

\theoremstyle{remark}
\newtheorem{remark}[theorem]{Remark}
\newtheorem{remarks}[theorem]{Remarks}

\def\gbf#1{\mbox{\boldmath$#1$}} 
\def\Mcal{{\mathcal M}} 
\def\cyc{{\rm cyc}}

\newcommand{\kk}{\ensuremath{\Bbbk}} 

\newcommand{\CC}{\ensuremath{\mathbb{C}}}

\newcommand{\QQ}{\ensuremath{\mathbb{Q}}} 
\newcommand{\RR}{\ensuremath{\mathbb{R}}} 

\newcommand{\ZZ}{\ensuremath{\mathbb{Z}}} 

\newcommand{\one}{\ensuremath{(\mathrm{i})}}
\newcommand{\two}{\ensuremath{(\mathrm{ii})}}
\newcommand{\three}{\ensuremath{(\mathrm{iii})}}
\newcommand{\four}{\ensuremath{(\mathrm{iv})}}

\newcommand{\coh}{\operatorname{coh}}

\newcommand{\head}{\operatorname{\mathsf{h}}}

\newcommand{\im}{\operatorname{im}}

\newcommand{\lcm}{\operatorname{lcm}}
\newcommand{\ltensor}{\overset{\mathbf{L}}{\otimes}}
\newcommand{\op}{\operatorname{op}}

\def\qpol{\mathrm{Q}}

\newcommand{\soc}{\operatorname{soc}}
\newcommand{\supp}{\operatorname{supp}}
\newcommand{\tail}{\operatorname{\mathsf{t}}}

\newcommand{\End}{\operatorname{End}}

\newcommand{\Ext}{\operatorname{Ext}} 
\newcommand{\GHilb}{\ensuremath{G}\operatorname{-Hilb}}

\newcommand{\Hom}{\operatorname{Hom}}

\newcommand{\SL}{\operatorname{SL}} 
\newcommand{\Spec}{\operatorname{Spec}}

 \newcommand{\modA}{\operatorname{mod-}\ensuremath{\!A}}
 \newcommand{\modAop}{\operatorname{mod-}\ensuremath{\!A^{\op}}}

\newcommand{\Span}{\operatorname{span}}
\newcommand{\cosup}{\operatorname{cosupp}}
\newcommand{\cL}[2]{{{#1 \widehat A #2}}}

\newcommand{\vtx}[1]{*+[o][F-]{\scriptscriptstyle #1}}

\title{Geometric Reid's recipe for dimer models}

 \author{Raf Bocklandt, Alastair Craw and Alexander Quintero V\'{e}lez} 

\address{Korteweg de Vries Instituut voor Wiskunde \\
Universiteit van Amsterdam\\
P.O. Box 94248\\
1090 GE Amsterdam}
 \email{raf.bocklandt@gmail.com}
 
\address{Department of Mathematical Sciences\\ University of Bath\\ Bath BA2 7AY, UK}
 \email{A.Craw@bath.ac.uk}
 
 \address{Instituto de Matem\'{a}ticas, Universidad de Antioquia, Calle 67 No. 53--108, Medell\'{i}n, Colombia}
 \email{a.quinterovelez@ciencias.udea.edu.co}

\begin{document}
\bibliographystyle{plain}

\begin{abstract}
Crepant resolutions of three-dimensional toric Gorenstein singularities are derived equivalent to noncommutative algebras arising from consistent dimer models. By choosing a special stability parameter and hence a distinguished crepant resolution $Y$,  this derived equivalence generalises the Fourier-Mukai transform relating the $G$-Hilbert scheme and the skew group  algebra $\CC[x,y,z]\ast G$ for a finite abelian subgroup of $\SL(3,\CC)$. We show that this equivalence sends the vertex simples to pure sheaves, except for the zero vertex which is mapped to the dualising complex of the compact exceptional locus. This generalises results of Cautis--Logvinenko~\cite{CautisLogvinenko09} and Cautis--Craw--Logvinenko~\cite{CautisCrawLogvinenko12} to the dimer setting, though our approach is different in each case. We also describe some of these pure sheaves explicitly and compute the support of the remainder, providing a dimer model analogue of results from Logvinenko~\cite{Logvinenko10}.
\end{abstract}
    
 \maketitle
 
\section{Introduction}

\subsection{Background}
The McKay correspondence for a finite subgroup $G\subset \SL(2,\CC)$ is a bijection between the irreducible representations of $G$ and an integral  basis of the cohomology on the minimal resolution $Y$ of $\CC^2/G$. Building on the original geometric construction by Gonzalez-Sprinberg--Verdier~\cite{GonzalezSprinbergVerdier83}, Kapranov--Vasserot~\cite{KapranovVasserot00} deduced the bijection from a derived equivalence
\begin{equation}
\label{eqn:functorIntro}
\Psi\colon D^b(\modA)\longrightarrow D^b(\coh(Y))
\end{equation}
between the bounded derived category of finitely generated modules over the skew group algebra $A=\CC[x,y]*G$ and the bounded derived category of coherent sheaves on $Y$. The equivalence sends the simple $A$-modules associated to the nontrivial irreducible representations of $G$ to pure sheaves supported on the irreducible components of the exceptional fibre of the resolution. For a finite subgroup $G\subset \SL(3,\CC)$, Bridgeland--King--Reid~\cite{BKR01} established a derived equivalence of the form \eqref{eqn:functorIntro} between the skew group algebra $A=\CC[x,y,z]\ast G$ and a distinguished choice of crepant resolution of $\CC^3/G$, namely the $G$-Hilbert scheme $Y=\GHilb(\CC^3)$. Further work of Bridgeland \cite{Bridgeland02} determines a derived equivalence for any other projective crepant resolution, but the $G$-Hilbert scheme nevertheless provides a particularly nice choice. 

One illustration of the special nature of $\GHilb(\CC^3)$ is provided by the \emph{Geometric McKay correspondence} of Cautis--Logvinenko~\cite{CautisLogvinenko09}. When $\CC^3/G$ has a single isolated singularity (so $G$ is abelian), the equivalence $\Psi$ sends the simple $A$-module $S_\rho$ associated to a nontrivial irreducible representation $\rho$ of $G$ to a pure sheaf on $Y$; this statement is false in general for $Y\neq \GHilb(\CC^3)$. Subsequent work of Logvinenko~\cite{Logvinenko10} showed that the support of the pure sheaf $\Psi(S_\rho)$ is either an irreducible exceptional divisor, a single $(-1,-1)$-curve or a chain of exceptional divisors in $\GHilb(\CC^3)$ and, moreover, the locus is determined by the role that $\rho$ plays in \emph{Reid's recipe} (see \cite[Section~6]{Reid97} and \cite[Section~3]{Craw05}). More generally, for any finite abelian subgroup $G\subset \SL(3,\CC)$, Cautis--Craw--Logvinenko~\cite{CautisCrawLogvinenko12} use Reid's recipe to compute explicitly the object $\Psi(S_\rho)$ for any irreducible representation $\rho$ of $G$. The goal of the present paper is to generalise the results from \cite{CautisLogvinenko09, Logvinenko10}, in addition to some results from \cite{CautisCrawLogvinenko12}, to any Gorenstein toric 3-fold singularity.

\subsection{Main results}
A dimer model is a polygonal cell decomposition of a real two-torus where each edge is oriented in such a way that the cycle on the boundary of each face is oriented. The vertices and oriented edges define a quiver $\qpol$, while the cycles around the boundary of each face give the terms (up to sign) in a potential $W$, and the Jacobian algebra $A$ is defined to be the quotient of the path algebra of $\qpol$ by the ideal of relations generated by the formal cyclic derivatives of $W$ with respect to the arrows in $\qpol$. If the dimer model satisfies a consistency condition, then $A$ is a noncommutative crepant resolution of its centre $Z(A)$, which is the coordinate ring of a 3-dimensional toric Gorenstein singularity. 

Every projective, crepant resolution of the singularity $X=\Spec Z(A)$ is obtained as a fine moduli space of $\theta$-stable representations of $A$ with dimension vector $(1,\dots,1)$. While any generic stability parameter $\theta$ will do, here we fix once and for all a vertex $0$ of $\qpol$ and define a special parameter $\vartheta$ so that an $A$-module is $\vartheta$-stable if and only if it is cyclic with generator $0$. The resulting moduli space $Y=\mathcal{M}_\vartheta$ comes with a tautological bundle of $\vartheta$-stable representations $T$, and the dual bundle $T^\vee = \mathcal{H}om_{\mathscr{O}_Y}(T, \mathscr{O}_Y)$ determines an equivalence of derived categories
\begin{equation}
\label{eqn:introPsi}
\Psi(-):= T^\vee\ltensor_{A} - \;\;\colon D^b(\modA) \longrightarrow D^b(\coh(Y)).
\end{equation}
In the special case where the dimer model tiles the torus with triangles, then $A$ is the skew group algebra for a finite abelian subgroup 
$G \subset \SL(3,\CC)$ as in \cite[Section~10.2]{CrawIshii04}. If our chosen vertex $0$ in $\qpol$ comes from the trivial representation of $G$, then $Y$ is the $G$-Hilbert scheme and \eqref{eqn:introPsi} coincides with the derived equivalence \eqref{eqn:functorIntro} from the McKay correspondence.

Our first main result generalises \cite[Theorem~1.1]{CautisLogvinenko09} and \cite[Proposition~5.6]{CautisCrawLogvinenko12} to dimer models. To state this result, let $S_i$ denote the simple $A$-module corresponding to vertex $i$ in $\qpol$. Also, let $\tau\colon Y\to X$ denote the crepant resolution and write $x_0\in X$ for the unique torus-invariant point.

\begin{theorem}
\label{thm:intromain1}
Let $\qpol$ be a consistent dimer model. Then:
\begin{enumerate}
\item[\one] for any vertex $i\neq 0$, the object $\Psi(S_i)$ is a pure sheaf, that is, it is quasi-isomorphic to a shift of a coherent sheaf; and 
\item[\two] the derived dual $\Psi(S_0)^\vee$ is quasi-isomorphic to the pushforward of $\mathscr{O}_{\tau^{-1}(x_0)}$ shifted by $3$. 
\end{enumerate}
\end{theorem}

\begin{corollary}
\label{cor:0vertex}
The object $\Psi(S_0)$ is quasi-isomorphic to the dualizing complex of $\tau^{-1}(x_0)$.  In particular, $\Psi(S_0)$ is a pure sheaf if and only if $\tau^{-1}(x_0)$ is equidimensional.
\end{corollary}

Our methods of proof differ from those in both \cite{CautisLogvinenko09} and \cite{CautisCrawLogvinenko12}. The first step presents the object $\Psi(S_i)$ explicitly by translating the standard projective resolution of the simple module $S_i$ into a complex of locally free sheaves on $Y$. For the zero vertex case, we use a description of $T^\vee$ in terms of weak paths in the quiver to write $\Psi(S_0)$ as a complex of sheaves of weak paths, after which we need only calculate the cohomology on each toric chart to deduce Theorem~\ref{thm:intromain1}\two, from which Corollary~\ref{cor:0vertex} follows  directly.

Our proof of Theorem~\ref{thm:intromain1}\one\ involves a geometric argument. Using the description of $\Psi(S_i)$ as a complex of locally free sheaves on $Y$, we compute the cohomology sheaves using the main result from Craw--Quintero-V\'{e}lez~\cite[Theorem~1.1]{CrawQuinterovelez12}. Our choice of the special stability parameter $\vartheta$ implies that the cohomology sheaves $H^k(\Psi(S_i))$ vanish for all $k\neq -1,0$. It remains therefore to show that if one of these cohomology sheaves is nonzero then the other must vanish. We achieve this by establishing the following link between objects $\Psi(S_i)$ that have nonvanishing cohomology in degree zero and certain walls of the GIT chamber containing the stability parameter $\vartheta$ (see Propositions~\ref{prop:TFAE} and \ref{prop:type0or1wall}):

\begin{proposition}
\label{prop:introwall}
Let $i\in \qpol_0$ be a nonzero vertex. The following are equivalent:
\begin{enumerate}
\item[\one] the sheaf $H^0(\Psi(S_i))$ is nonzero;
\item[\two] there exists a torus-invariant $\vartheta$-stable $A$-module that contains $S_i$ in its socle;
\item[\three] $\overline{C}\cap S_i^\perp$ is a wall of the chamber $C$ containing $\vartheta$.
\end{enumerate}
In this case, $\Psi(S_i)\cong L_i^{-1}\vert_{Z_i}$,
where $Z_i$ is the unstable locus of the wall $\overline{C}\cap S_i^\perp$.
\end{proposition}
Proposition~\ref{prop:introwall} implies that an object $\Psi(S_i)$ with nonvanishing cohomology in degree zero has vanishing cohomology in every other degree, so Theorem~\ref{thm:intromain1}\one\ follows. We also prove (see Lemma~\ref{lem:0orI}) that the unstable locus $Z_i$ of a wall of the form $\overline{C}\cap S_i^\perp$ is either a single $(-1,-1)$-curve or a connected union of compact torus-invariant divisors. In addition, we make two further remarks about this result:
\begin{itemize}
\item The description of the support of the objects from Proposition~\ref{prop:introwall} in terms of modules with a given vertex simple module in their socle extends to dimension three the characterisation by Crawley--Boevey~\cite[Theorem~2]{Crawley-Boevey00} of the exceptional curves in the minimal resolution of $\CC^2/G$ for a finite subgroup $G\subset \SL(2,\CC)$. 
\item Proposition~\ref{prop:introwall} provides a straightforward method for computing the list of vertices for which $\Psi(S_i)$ has cohomology concentrated in degree zero: one can readily compute the set of torus-invariant $\vartheta$-stable $A$-modules, and one need only inspect whether $S_i$ lies in the socle of one such module to decide whether $\overline{C}\cap S_i^\perp$ is a wall of the chamber $C$ and hence whether $H^0(\Psi(S_i))$ is nonzero.
\end{itemize}

Our results thus far say nothing about the object $\Psi(S_i)$ when its cohomology is concentrated in degree $-1$. In fact, $H^{-1}(\Psi(S_i))$ is rather complicated in general (see Proposition~\ref{prop:H-1}), and one benefit of our proof of Theorem~\ref{thm:intromain1} is that we avoid having to analyse this sheaf. To conclude, however, we adapt a result of Logvinenko~\cite{Logvinenko10} in order to prove that if the sheaf $H^{-1}(\Psi(S_i))$ is nonzero, then its support is a connected union of compact, torus-invariant divisors. Together with the above results, this leads to the following dimer model analogue of the \emph{Geometric McKay correspondence in dimension three} proven by Logvinenko~\cite[Theorem~1.1]{Logvinenko10}:

\begin{theorem}[Geometric Reid's recipe for dimer models]
\label{thm:intromain2}
Let $\qpol$ be a consistent dimer model and let $i$ be a vertex of $\qpol$. Then $\Psi(S_i)$ is quasi-isomorphic to one of the following:
\begin{enumerate}
 \item[\one] $L_i^{-1}$ restricted to a connected union of compact irreducible torus-invariant divisors;
 \item[\two]  $L_i^{-1}$ restricted to a compact torus-invariant curve; 
 \item[\three] $\mathcal{F}[1]$ for a sheaf $\mathcal{F}$ supported on a connected union of compact torus-invariant divisors;  
 \item[\four] the dualising complex of the compact exceptional locus $\tau^{-1}(x_0)$. 
\end{enumerate}
\end{theorem}

As we saw above, the torus-invariant divisor or curve appearing in Theorem~\ref{thm:intromain2}\one-\two\ is the unstable locus $Z_i$ of a GIT wall of the form $\overline{C}\cap S_i^\perp$

\subsection{Future directions}
It is natural to seek a precise description of the pure sheaves $\Psi(S_i)$ in degree $-1$ from Theorem~\ref{thm:intromain2}\three. However, even for the McKay quiver case, the description of the pure sheaves that lie in degree $-1$ from Cautis--Craw--Logvinenko~\cite[Theorem~1.2]{CautisCrawLogvinenko12} relies on the original combinatorial version of Reid's recipe (see \cite[Section~6]{Reid97} and \cite[Section~3]{Craw05}) that associates a nontrivial irreducible representation of $G$ to every compact torus-invariant curve and surface in $\GHilb(\CC^3)$. In fact, even the precise description of the support of these objects by Logvinenko~\cite[Theorem~1.2]{Logvinenko10} was phrased in terms of combinatorial Reid's recipe. The combinatorial version of Reid's recipe for dimer models is the subject of forthcoming work by Tapia Amador~\cite{Tapiaamador14}, where nonzero vertices of $\qpol$ are associated to every compact torus-invariant curve and surface in $Y=\mathcal{M}_\vartheta$. The results thus far are compatible with Theorem~\ref{thm:intromain2}: a vertex $i$ in $\qpol$ labels a compact torus-invariant surface $S\subset Y$ if and only if $S$ is contained in the support of $\Psi(S_i)$ as in Theorem~\ref{thm:intromain2}\one; and $i$ labels a unique compact torus-invariant curve $\ell\subset Y$ if and only if $\ell$ coincides with the support of $\Psi(S_i)$ as in Theorem~\ref{thm:intromain2}\two. We anticipate that the results from \cite{Tapiaamador14} will lead in due course to a precise description of the pure sheaves $\Psi(S_i)$ in degree $-1$, but the details from \cite{CautisCrawLogvinenko12} suggest that this will not be straightforward.

As a final remark, recall Conjecture~1.2 from Cautis--Logvinenko~\cite{CautisLogvinenko09}: \emph{for a finite, non-abelian subgroup $G\subset \SL(3,\CC)$,  the objects $\Psi(S_\rho)$ on the $G$-Hilbert scheme arising from any (nontrivial) representation $\rho$ of $G$ are pure sheaves}. Proposition~\ref{prop:introwall} suggests that the following addition to this conjecture might be true:

\begin{conjecture}
The object $\Psi(S_\rho)$ is a pure sheaf in degree 0 if and only if $\overline{C}\cap S_\rho^\perp$ is a wall of the chamber $C$ defining $\GHilb(\CC^3)$, in which case $\Psi(S_\rho)\cong L_\rho^{\vee}\vert_{Z_\rho}$ where $Z_\rho$ is the unstable locus of the wall $\overline{C}\cap S_\rho^\perp$.
\end{conjecture}

 Certainly, an improved understanding of variation of GIT quotient for $G$-Hilb looks to be important for extending Reid's recipe to the non-abelian case.

  \medskip

 \noindent \textbf{Acknowledgements.} 
 The second author thanks Akira Ishii and Kazushi Ueda for a preliminary draft of their preprint \cite{IshiiUeda13}. Thanks also to Alastair King for a useful remark, and to the anonymous referee for many helpful comments. The second and third authors were supported in part by EPSRC grant EP/G004048/1.
 
\section{Preliminaries}
In this section we present standard definitions and results from the literature to establish our notation. Throughout, $\kk$ denotes an algebraically closed field of characteristic zero.

\subsection{Quivers and superpotentials}
\label{Sect:2.1}
 A \emph{quiver} $\qpol$ is an oriented graph given by a finite set of vertices $\qpol_0$, a finite set of arrows $\qpol_1$, and maps $\head, \tail \colon \qpol_1 \to \qpol_0$ indicating the vertices at the head and tail of each arrow. We assume throughout that the underlying graph is connected. A path in $\qpol$ of length $k\geq 1$ is a sequence $p=a_k \cdots a_1$ of arrows such that $\head(a_{j})=\tail(a_{j+1})$ for $1\leq j < k$. We set $\head(p)=\head(a_k)$, $\tail(p)=\tail(a_1)$. In addition, each vertex $i \in \qpol_0$ determines a trivial path $e_i$ with $\head(e_i)=\tail(e_i)=i$. A cycle in $\qpol$ is a nontrivial path $p$ in which $\head(p)=\tail(p)$. The \emph{path algebra} $\kk \qpol$ associated to $\qpol$ is the $\kk$-algebra whose underlying vector space has basis the set of paths in $\qpol$, where the product of paths is given by the concatenation if possible, and zero otherwise. A \emph{relation} for the quiver $\qpol$ is a non-zero $\kk$-linear combination of paths of length at least $2$ having the same head and tail. If $\rho$ is a set of relations on $\qpol$, the pair $(\qpol,\rho)$ is called a quiver with relations. Associated with $(\qpol, \rho)$ is the $\kk$-algebra $A = \kk \qpol/ \langle \rho \rangle$, where $\langle \rho \rangle$ denotes the two-sided ideal in $\kk \qpol$ generated by the set of relations $\rho$. 

Let $\qpol$ be a quiver with cycles and let $[\kk \qpol,\kk \qpol]$ denote the $\kk$-vector space spanned by all commutators in $\kk \qpol$. The quotient $\kk \qpol_{\cyc} := \kk \qpol/[\kk \qpol,\kk \qpol]$ has a basis corresponding to all cycles in $\qpol$ up to cyclic shifts. Elements of $\kk \qpol_{\cyc}$ are called \emph{superpotentials} of $\qpol$. For $a \in \qpol_1$ and for each cycle $p = a_k \cdots a_1$ in $\qpol$, define the cyclic derivative $\partial_a p$ by setting
$$
\partial_a p:= \sum_{i=1}^k \delta_{a,a_i} a_{i-1} \cdots a_1 a_k \cdots a_{i+1},
$$
where $\delta_{a,a_i}$ is the Kronecker delta. Extending this by linearity to $\kk \qpol_{\cyc}$ defines a $\kk$-linear map $\partial_a\colon \kk \qpol_{\cyc} \to \kk \qpol$. Given a superpotential $W \in \kk \qpol_{\cyc}$, each element $\partial_a W\in \kk \qpol$ is a $\kk$-linear combination of paths in $\qpol$ that share the same head and tail, so we obtain an ideal of relations $\langle \partial_a W \mid a \in \qpol_1 \rangle$. The \emph{Jacobian algebra} of the quiver $\qpol$ with superpotential $W$ is defined to be 
$$
A := \kk \qpol/\langle \partial_a W \mid a \in \qpol_1 \rangle.
$$
In this article we consider only a particularly well behaved class of quivers with superpotentials that arise from consistent dimer models.

\subsection{Consistent dimer models}
A \emph{dimer model} is a quiver $\qpol$ whose underlying graph provides a polygonal cell decomposition of the surface of a real two-torus, such that the cycle on the boundary of each face is oriented and has length at least $3$. Assume in addition that every face of the \emph{dual} cell decomposition is simply connected.  Let $\qpol_2$ denote the set of faces. For any face $F \in \qpol_2$, let $W_F \in \kk Q_{\cyc}$ be the cycle obtained by tracing all arrows around the boundary of $F$. Define the \emph{superpotential} for $\qpol$ to be
$$
W:= \sum_{F \in Q_2} (-1)^F W_F,
$$
where $(-1)^F$ takes value $+1$ on faces oriented anticlockwise, and $-1$ on faces oriented clockwise. We write $A$ for the Jacobian algebra of the quiver $\qpol$ with superpotential $W$.

The Jacobian algebra has a special central element $\omega=\sum_{i\in Q_0} c_i$ where $c_i$ is a clockwise boundary cycle starting at $i$. The relations in $\langle \partial_a W \mid a \in \qpol_1 \rangle$ ensure that $\omega$ is central in $A$ and does not depend on the particular choice of cycles $c_i$. To a dimer model we can also associate its \emph{weak Jacobian algebra} by inverting $\omega$:
\[
\widehat A := A \otimes_{\kk[\omega]}\kk[\omega,\omega^{-1}].
\]
Equivalently, this algebra can be obtained by inverting all the arrows. More precisely, define the double quiver $\widehat \qpol$ to be the quiver obtained from $\qpol$ by introducing an extra arrow $a^{-1}$ in the opposite direction for each $a\in \qpol_1$. A path in the quiver $\widehat \qpol$ is called a \emph{weak path}. Then
\[
\widehat A = \kk \widehat \qpol/\big\langle \partial_a W,\; aa^{-1}-e_{\head(a)}, a^{-1}a-e_{\tail(a)} \mid a \in \qpol_1\big\rangle.
\]
A dimer model is said to be \emph{consistent} if the standard localisation map $A \to \widehat A$ is injective.  

 \begin{remark}
 The literature contains several notions of consistency for dimer models. We work with the cancellation property of Mozgovoy--Reineke~\cite{MozgovoyReineke10} and studied further by Davison~\cite{Davison11}. Bocklandt~\cite{Bocklandt12} showed that the cancellation property is equivalent both to algebraic consistency introduced by Broomhead~\cite{Broomhead12}, and to the existence of a `consistent R-charge' as studied by Kennaway~\cite{Kennaway13}. Moreover, \cite[Lemma~7.4]{Bocklandt12} shows that every such dimer model $\qpol$ carries a perfect matching, in which case Ishii--Ueda~\cite{IshiiUeda10} implies that the cancellation property is equivalent both to $\qpol$ being property ordered in the sense of Gulotta~\cite{Gulotta08} and to the notion of consistency from Ishii--Ueda~\cite[Definition~3.5]{IshiiUeda10}. In short, all of the above notions of consistency are equivalent. 
 \end{remark}

The following result is due to Broomhead~\cite[Lemma~5.6]{Broomhead12}.

\begin{proposition}
\label{prop:broomhead}
For each (algebraically) consistent dimer model $\qpol$, the centre $Z(A)$ of the corresponding Jacobian algebra $A$ is a Gorenstein semigroup algebra.
\end{proposition}

\subsection{Moduli of quiver representations}
A \emph{representation} $V$ of a quiver $\qpol$ is a collection $\{V_i \mid i \in \qpol_0 \}$ of finite dimensional $\kk$-vector spaces and a collection $\{v_a \colon V_{\tail(a)} \to V_{\head(a)} \mid a \in \qpol_1\}$ of $\kk$-linear maps. The dimension vector of a representation $V$ is the vector $\dim(V) \in \ZZ^{\qpol_0}$ whose $i$th component is $\dim (V_i)$ for $i\in \qpol_0$. It is convenient to write a representation $V$ of dimension vector $\underline{1}:=(1,1,\dots, 1)$ simply as the tuple of scalars $V=(v_a)_{a\in \qpol_1}$. We also write $v_p:=v_{a_1}\dots v_{a_n}$ for a path $p=a_1\dots a_n$ and $v_{\sum \lambda_i p_i}:=\sum \lambda_i v_{p_i}$ for a linear combination of paths with the same head and tail.

If $V, V^\prime$ are two representations of $\qpol$, then a morphism $f \colon V \to V^\prime$ is a collection of $\kk$-linear maps $\{f_i \colon V_i \to V^\prime_i \mid i \in \qpol_0 \}$ such that $v^\prime_a f_{\tail(a)} = f_{\head(a)}v_a$ for all $a \in \qpol_1$. In the presence of relations, a representation of the pair $(\qpol,\rho)$ is a representation $V$ of $\qpol$ such that for each relation $r \in \rho$, the corresponding linear combination of $\kk$-linear maps from $V_{\tail(r)}$ to $V_{\head(r)}$ is zero. The abelian category of representations of $(\qpol,\rho)$ is equivalent to the category of finite-dimensional left modules over the quotient algebra $A= \kk \qpol/ \langle \rho \rangle$, and we let $\modA$ denote this category. Each vertex $i \in \qpol_0$ defines a simple object $S_i=\kk e_i$ in $\modA$ called the \emph{vertex simple} for $i\in \qpol_0$; as a representation of $\qpol$, this has $V_i=\kk$ and $V_j=0$ for $j\neq i$, where the maps $f_a$ are zero for all $a\in \qpol_1$. In general these are not the only simple objects of $\modA$.

Consider the rational vector space
\[
\Theta = \Big\{\theta \in \Hom\big(\ZZ^{\qpol_0},\QQ\big) \mid \textstyle{\sum_{i\in \qpol_0} \theta_i = 0}\Big\}.
\]
Given a representation $V$ of $\qpol$, define $\theta(V):=\theta(\dim V)$ for $\theta\in \Theta$. A representation $V$ of dimension vector $\underline{1}=(1,1,\dots, 1)$ is \emph{$\theta$-semistable} if every subrepresentation $V^\prime\subset V$ satisfies $\theta(V^\prime)\geq 0$, and it is \emph{$\theta$-stable} if $\theta(V^\prime) > 0$ for every nonzero, proper subrepresentation $V^\prime\subset V$. In the presence of relations and for the algebra $A= \kk \qpol/ \langle \rho \rangle$, the equivalence of abelian categories from above gives us a notion of $\theta$-stability for $A$-modules.  We say that $\theta\in \Theta$ is \emph{generic} if every $\theta$-semistable representation is $\theta$-stable. For $\theta\in \Theta$, King~\cite{King94} constructs the coarse moduli space $\overline{\mathcal{M}_\theta}$ of S-equivalence classes of $\theta$-semistable $A$-modules of dimension vector $\underline{1}$ using geometric invariant theory. Our primary interest lies 
with generic $\theta\in \Theta$, in which case $\overline{\mathcal{M}_\theta}$ coincides with the fine moduli space $\mathcal{M}_\theta$ of isomorphism classes of $\theta$-stable $A$-modules of dimension vector $\underline{1}$.  

Since $\mathcal{M}_\theta$ represents a functor, it carries a universal family of $\theta$-stable $A$-modules of dimension vector $\underline{1}$, namely a tautological locally free sheaf 
\[
 T=\bigoplus_{i\in \qpol_0} L_i,
\] 
where $L_i$ has rank one for all $i\in \qpol_0$, together with a tautological $\kk$-algebra homomorphism 
\[
\phi\colon A\rightarrow \End_{\mathscr{O}_{\mathcal{M}_\theta}}(T).
\]
For each closed point $y\in \mathcal{M}_\theta$, let $V_y$ denote the $\theta$-stable $A$-module obtained by restricting the tautological maps $\phi(a)\colon L_{\tail(a)}\to L_{\head(a)}$ for $a\in \qpol_1$ to the fibre of the tautological bundle $T$ over the point $y\in \mathcal{M}_\theta$. As King~\cite[Proposition~5.3]{King94} notes, there is an inherent ambiguity in the construction of $T$. To remove this ambiguity, we distinguish once and for all a vertex of the quiver that we denote $0 \in Q_0$, and we normalise the tautological bundle by fixing $L_0\cong \mathscr{O}_{\mathcal{M}_\theta}$.

\subsection{Crepant resolutions and tilting bundles}
\label{Sect:3.1}
In a series of papers, Ishii--Ueda~\cite{IshiiUeda08, IshiiUeda09,IshiiUeda13} study certain fine moduli spaces $\mathcal{M}_\theta$ associated to the Jacobian algebra $A$ defined by a consistent dimer model $\qpol$. Here we recall some of their main results.

We begin with the main geometric result of \cite{IshiiUeda08}; our statement of Theorem~\ref{thm:ishiiueda08} differs slightly from that in the original paper, see Remark~\ref{rem:ishiiueda08} below. 

 \begin{theorem}[Ishii--Ueda~\cite{IshiiUeda08}]
 \label{thm:ishiiueda08}
 Let $A$ be the Jacobian algebra of a consistent dimer model and write $X=\Spec Z(A)$ for the Gorenstein toric variety of dimension three (see Proposition~\ref{prop:broomhead}). For $\theta\in \Theta$ generic, the toric variety $\mathcal{M}_\theta$ is a crepant resolution of $X$.
 \end{theorem}

\begin{remark}
\label{rem:ishiiueda08}
Ishii--Ueda~\cite{IshiiUeda08} assume only that the dimer model is nondegenerate in order to deduce that $\mathcal{M}_\theta$ is a crepant resolution for generic $\theta\in \Theta$ (consistent dimer models are shown to be nondegenerate in \cite[Proposition~7.1]{IshiiUeda09}). In fact, they prove that $\mathcal{M}_\theta$ is a crepant resolution of (the normalisation of) the closure $X^\prime$ of a three-torus $(\kk^*)^3$ in $\overline{\mathcal{M}_0}$. Craw--Quintero-V\'{e}lez~\cite{CrawQuinterovelez12} subsequently gave a GIT construction of a toric subvariety $Y_\theta\subseteq \overline{\mathcal{M}_\theta}$ for arbitrary $\theta\in \Theta$ such that $Y_\theta = \mathcal{M}_\theta$ for generic $\theta\in \Theta$, and $X\cong Y_0\cong X^\prime$. In particular, $X^\prime$ is normal, and for generic $\theta\in \Theta$ we obtain a commutative diagram
 \[
 \begin{CD}   
  Y_\theta @= \mathcal{M}_\theta \\
     @V{\tau_\theta}VV   @VVV    \\
X @>>> \overline{\mathcal{M}_0} \\
 \end{CD}
 \]
 where the horizontal map is a closed immersion and the vertical maps are projective morphisms obtained by variation of GIT quotient. As a result, the morphism $\tau\colon \mathcal{M}_\theta \to X^\prime$ from \cite{IshiiUeda08} coincides with the morphism $\tau_\theta\colon Y_\theta\to X$ obtained by variation of GIT quotient. 
\end{remark}

Next we recall two important results that link properties of $A$ with the geometry of $\Mcal_{\theta}$.

 \begin{theorem}[Ishii--Ueda~\cite{IshiiUeda09,IshiiUeda13}]
 \label{thm:ishiiueda09}
 Let $A$ be the Jacobian algebra of a consistent dimer model. For $\theta\in \Theta$ generic, write $\tau\colon \mathcal{M}_\theta\to X=\Spec Z(A)$ for the crepant resolution and $T$ for the tautological bundle on $\Mcal_{\theta}$.  Then:
   \begin{enumerate}
 \item[\one] the tautological $\kk$-algebra homomorphism $\phi\colon A\rightarrow \End_{\mathscr{O}_{\mathcal{M}_\theta}}(T)$ is an isomorphism; and
 \item[\two] the bundle $T$ on $\Mcal_{\theta}$ is a tilting bundle.
 \end{enumerate}
 \end{theorem}
 
 \begin{remarks}
\label{rem:ishiiueda09}
\begin{enumerate}
 \item For $i, j\in \qpol_0$, the space $e_j A e_i$ is spanned by classes of paths in $\qpol$ from $i$ to $j$. Under the isomorphism from Theorem~\ref{thm:ishiiueda09}\one, the image of $e_jAe_i$ is $\Hom(L_i,L_j)$ which, by an unfortunate consequence of the notation for $\Hom$ spaces, is $e_j \End(T)e_i$. Thus, $H^0(L_j)$ is isomorphic to the space $e_j A e_0$ of classes of paths in $\qpol$ from 0 to $j$. This also implies that the centre of $A$ is
 \begin{equation}
 \label{eqn:centre}
 Z(A)\cong e_jA e_j\cong\End_{\mathscr{O}_{\mathcal{M}_\theta}}(L_j)\cong \kk[X]
 \end{equation}
for each $j\in \qpol_0$. These observations are generalised in Lemma~\ref{lem:widehat}.
\item Let $F\in \qpol_2$ be a face and let $i\in \qpol_0$ be any vertex through which the cycle $W_F\in \kk\qpol_{\cyc}$ passes. The zero-locus of the map $\phi(W_F)\colon L_i\to L_i$ is the union of all torus-invariant divisors in $\mathcal{M}_\theta$, each with multiplicity one, see \cite[Lemma~4.1]{IshiiUeda09}.  
\item That $T$ is a tilting bundle means in particular that the functors 
 $$
\mathbf{R}\!\Hom_{\mathscr{O}_{\Mcal_{\theta}}}(T,-)\colon D^b(\coh(\mathcal{M}_{\theta})) \longrightarrow D^b(\modAop)
$$
and
 \[
-\ltensor_{A}T \colon D^b(\modAop) \longrightarrow D^b(\coh(Y))
\]
provide mutually inverse equivalences between the bounded derived category of coherent sheaves on $\mathcal{M}_{\theta}$ and the bounded derived category of finite-dimensional left $A^{\mathrm{op}}$-modules. 
\end{enumerate}
\end{remarks}
    
\begin{example}
\label{exa:McKay}
Let $\qpol$ denote the McKay quiver of a finite abelian subgroup $G\subset \SL(3,\CC)$, and let $A=\CC[x,y,z]*G$ be the skew group algebra. The vertex set $\qpol_0$ is the set of isomorphism classes of irreducible representations of $G$. If we choose the zero vertex to be the trivial representation, then Ito--Nakajima~\cite{ItoNakajima00} observed that the fine moduli space $\mathcal{M}_\vartheta$ of $\vartheta$-stable $A$-modules coincides with Nakamura's $G$-Hilbert scheme for any $\vartheta\in \Theta$ satisfying $\vartheta_i>0$ for $i\neq 0$. 
\end{example}

\subsection{An explicit description of $\mathcal M_\theta$ and $T$}
The geometry of $\mathcal{M}_\theta$ is determined by the combinatorial data encoded by its toric fan. To construct this fan directly from the dimer we first need a pair of dual lattices $M,N$. In our case we take for $M$ the set of all classes of weak paths in $e_0\widehat A e_0$. If we let $x,y$ denote two weak paths that span the homology of the torus and let $z=e_0\omega e_0$,  then we can identify $M$ with $\ZZ x\oplus \ZZ y\oplus \ZZ z$ and $e_0\widehat A e_0$ with $\kk[M]$. The dual lattice $N$ may be regarded as the lattice of all possible $\ZZ$-gradings on $M$.

Recall that a {\em perfect matching} in $\qpol$ is a subset $\Pi\subseteq \qpol_1$ such that each face in $\qpol_2$ contains exactly one arrow from $\Pi$ on its boundary. To a perfect matching $\Pi$ we can associate a $\ZZ$-grading $\deg_\Pi$ on $A$ (and $\hat A$) determined by 
\[
\forall a \in \qpol_1:\deg_{\Pi} a := \begin{cases}
1 &a\in \Pi\\
0 & a \notin \Pi
\end{cases}
\]
In this way every perfect matching $\Pi$ defines an element in $N$, namely
\[
n_\Pi: M \to \ZZ : p \mapsto \deg_\Pi p.
\]
Note that different perfect matchings can give the same element. Since $z$ is a boundary cycle, each lattice point of the form $n_\Pi$ lies in the affine plane $\{n\in N \otimes_\ZZ\RR \mid \langle n,z\rangle=1\}$. The convex hull of these points in $N \otimes_\ZZ\RR$ is a lattice polygon $P$, and the cone $\sigma$ over this polygon defines the Gorenstein toric variety $X=\Spec \kk[\sigma^\vee\cap M]$.

To describe the fan $\Sigma$ of $\mathcal{M}_\theta$ in terms of perfect matchings, define the cosupport of a representation $V=(v_a)_{a\in \qpol_1}$ to be the set $\{ a \in \qpol_1 \mid  v_a=0\}$ of arrows on which $V$ is zero. A perfect matching is called $\theta$-stable if it is the cosupport of a $\theta$-stable representation. The next result is a rephrasing of Ishii--Ueda \cite[Lemmas 6.1 and 6.2]{IshiiUeda08} (see also Mozgovoy~\cite[Proposition 4.15]{Mozgovoy09}).

\begin{lemma}
\label{lem:fan}
Let $\theta\in \Theta$ be generic. For $1\leq r\leq 3$, the set of $r$-dimensional cones in the fan $\Sigma$ of $\mathcal{M}_\theta$ is precisely the set of cones of the form $\sum_{i=1}^r \RR^+ n_{\Pi_i}$, where $\Pi_1,\dots, \Pi_r$ are $\theta$-stable perfect matchings such that $\Pi_1\cup \dots \cup \Pi_r$ is the cosupport of a $\theta$-stable representation.
\end{lemma}

\begin{remarks}\label{rem:cosupp}
\begin{enumerate}
\item Let $\Sigma(1)$ denote the set of one-dimensional cones in the fan $\Sigma$ of $\mathcal{M}_\theta$. For $\rho \in \Sigma(1)$, we write $\Pi_\rho$ for the unique $\theta$-stable perfect matching corresponding to $\rho$, and $\deg_\rho := \deg_{\Pi_\rho}$ for the associated the degree function. 
\item Intersecting the fan $\Sigma$ with the polygon $P$ gives a regular triangulation of $P$ (see Figure~\ref{fig:Sigma} for an example). The $3$-dimensional cones give elementary triangles, the $2$-dimensional cones give line segments and the rays give lattice points. In particular, each lattice point in the triangulation of $P$ corresponds to a unique $\theta$-stable perfect matching.
\item For $Y:=\mathcal{M}_\theta$, each cone $\sigma\in \Sigma$ defines both a torus orbit and an open chart in $Y$, namely $Y_\sigma= \Hom_{sg}(\sigma^\perp,\kk^*)\cong (\kk^*)^{3 - \dim \sigma}$ and $U_\sigma= \Hom_{sg}(\sigma^\vee,\kk)$ in $\mathcal{M}_\theta$ respectively (here $\Hom_{sg}$ denotes semigroup homomorphisms). In our case these definitions become
\begin{eqnarray*}
\kk[Y_\sigma] &=& \Span_\kk \{ p \in e_0\widehat A e_0 \mid \deg_\rho p=0 \;\forall \rho \subset \sigma\},\\
\kk[U_\sigma] &=& \Span_\kk \{ p \in e_0\widehat A e_0 \mid \deg_\rho p\ge 0 \;\forall \rho \subset \sigma\}.
\end{eqnarray*}
In view of this it makes sense to define
\[
\widehat A_\sigma := \Span_\kk \{p \in \widehat A\mid \forall n_\Pi \in \sigma:\deg_\Pi(p)\ge 0\}.
\]
With this notation we have $\kk[U_\sigma] = e_0 \widehat A_\sigma e_0$.
\item Let $\sigma\in \Sigma$ be a cone of dimension $r$ generated by the rays through $\{n_{\Pi_i}\in N \mid 1\leq i\leq r\}$, and write $\cosup_\sigma:= \Pi_1\cup\dots\cup \Pi_r$. Then
$Y_\sigma$ parametrises the set of all $\theta$-stable representations 
with cosupport equal to $\cosup_\sigma$, while $U_\sigma$ parametrises the set of all $\theta$-stable representations with cosupport contained in $\cosup_\sigma$.
\end{enumerate}
\end{remarks}

To describe the tautological bundle $T=\bigoplus_{i\in \qpol_0} L_i$ on $Y=\mathcal{M}_\theta$, recall that the tautological $\kk$-algebra isomorphism $\phi\colon A\rightarrow \End_{\mathscr{O}_{\mathcal{M}_\theta}}(T)$ from Theorem~\ref{thm:ishiiueda09}\one\ associates to each path $p$ in $\qpol$ a section $D_p\in H^0\big(L_{\head(p)}\otimes L^{-1}_{\tail(p)}\big)$. Extend this assignment to every path in the double quiver $\widehat{\qpol}$ by setting $D_{a^{-1}}:=-D_a$, so that for any weak path $p$ in $\qpol$, the associated divisor is $D_p\in H^0(L_{\head(p)}\otimes L^{-1}_{\tail(p)})$. Since $L_0\cong \mathscr{O}_Y$,  the following result is immediate:

\begin{lemma}
The tautological bundle $T=\bigoplus_{i\in \qpol_0} L_i$ satisfies $L_j \cong \mathscr{O}_Y(D_{p})$ for each $j\in \qpol_0$, where $p$ is any weak path from vertex $0$ to vertex $j$.
\end{lemma}

The main result of Bender--Mozgovoy~\cite[Theorem~4.2]{BenderMozgovoy09} presents a formula for any such divisor $D_p$ in terms of the torus-invariant prime divisors in $Y$, namely
\[
D_p = \sum_{\rho\in \Sigma(1)} \deg_{\rho} p\; E_\rho,
\] 
where $E_\rho := \overline{Y_\rho}$ is the orbit closure associated to the ray $\rho\in \Sigma(1)$. Observe that $D_\rho$ is a principal divisor iff $p \in e_i\widehat{A} e_i$ for some $i\in \qpol_0$. Indeed, it's clearly principal if $p \in e_0\widehat{A} e_0=\kk[M]$, and if $p \in e_i\widehat{A} e_i$ then $D_{qpq^{-1}}=D_p$ for any weak path $q\in e_0\widehat{A} e_i$. Conversely, if $p$ is not a cyclic path then $D_p$ is not principal, but the divisors of two weak paths $p_1,p_2$ with the same head and tail differ by the principal divisor $D_{p_1p_2^{-1}}$. 

\begin{lemma}\label{lem:widehat}
Let $\sigma\in \Sigma$. For any vertex $i\in \qpol_0$, we have 
isomorphisms of vector spaces
\begin{eqnarray*}
 \Gamma(U_\sigma,L_i) &\cong & \Span_\kk \{u \in e_i \widehat A e_0 \mid \deg_\rho u\ge 0 \; \forall \rho\subset \sigma\} = e_i \widehat A_\sigma e_0\\
 \Gamma(U_\sigma,L_i^{-1}) &\cong & \Span_\kk \{u \in e_0 \widehat{A} e_i \mid \deg_\rho u\ge 0\; \forall \rho\subseteq\sigma\} = e_0 \widehat{A}_\sigma e_i.
 \end{eqnarray*}
which, using the identification $\kk[U_\sigma]=e_0 \widehat A_\sigma e_0$, become isomorphisms of $\kk[U_\sigma]$-modules.
\end{lemma}
\begin{proof}
Fix a weak path $p \in e_i \widehat A e_0$ then 
$L_i \cong \mathscr{O}_Y(D_{p})$.
By definition
$$
\Gamma(U_\sigma,L_i)= \Span \{ q \in e_0\widehat A e_0 |\deg_\rho q \ge \deg_\rho -p \; \forall \rho\subset \sigma\}.
$$
Multiplying everything with $p$ gives the required isomorphism.
The proof for $L_i^{-1}$ is similar.
\end{proof}

\section{Computing images of vertex simples}
In this section we introduce the crepant resolution and the derived equivalence that we study throughout this paper. We describe the images of the vertex simples in terms of the tautological line bundles and the higher quiver of the dimer model, and we describe explicitly the cohomology sheaves of this object.

\subsection{The special stability condition and the dual tilting bundle}
Let $\qpol$ be a consistent dimer model on a real 2-torus with superpotential $W$ and Jacobian algebra $A$. One of the goals of this paper is to generalise to the dimer setting existing results on the $G$-Hilbert scheme for a finite abelian subgroup $G\subset \SL(3,\CC)$. To do so, we consider a stability condition of the form 
\begin{equation}
\label{eqn:vartheta}
\vartheta = (\vartheta_i)\in \Theta \text{ satisfying }\vartheta_i>0 \text{ for }i\neq 0,
\end{equation}
and hence $\vartheta_0<0$. Every such $\vartheta$ is generic, and we let $Y:=\mathcal{M}_\vartheta$ denote the fine moduli space of $\vartheta$-stable $A$-modules of dimension vector $\underline{1}$. Theorem~\ref{thm:ishiiueda08} gives a projective, crepant resolution $\tau\colon Y\to X=\Spec Z(A)$. The next result characterises $\vartheta$-stability combinatorially.

\begin{lemma}
\label{lem:specialtheta}
Let $V$ be a representation of $\qpol$ defining an $A$-module of dimension vector $\underline{1}$, and let $x_0\in X$ denote the unique torus-invariant point. Then:
\begin{enumerate}
\item[\one] $V$ is $\vartheta$-stable iff $V$ admits a nonzero path from $0$ to every other vertex $i\in \qpol_0$; and
\item[\two] $y\in Y$ lies in $\tau^{-1}(x_0)$ if and only if all maps in $V_y$ with head at vertex $0$ are zero.
\end{enumerate}
\end{lemma}
\begin{proof}
Part \one\ is immediate from the definition of $\vartheta$-stability and the fact that the dimension vector is $\underline{1}$. For part \two, write $V_y=(v_a)_{a\in \qpol_1}$ for the $\vartheta$-stable representation corresponding to $y\in Y$. The centre of $A$ is isomorphic to $e_0 Ae_0$ by \eqref{eqn:centre}, so every central element defines a cycle $c$ passing through vertex 0. We have $\tau(y)=x_0$ if and only if for every any nontrivial central element in $A$, the corresponding cycle $c$ satisfies $v_c =0$. Choose any arrow $a\in \qpol_1$ with head at vertex $0$ and set $i=\tail(a)$. By part \one, $V_y$ admits a nonzero path $p$ from $0$ to $i$, so $v_a= 0$ if and only if the cycle $pa$ through vertex 0 satisfies $v_{pa}=0$. This proves the result.
\end{proof}

Theorem~\ref{thm:ishiiueda09} implies that the tautological bundle $T$ on $Y$ is tilting, but our interest lies primarily with the derived equivalence induced by the dual bundle $T^\vee = \mathcal{H}om_{\mathscr{O}_Y}(T, \mathscr{O}_Y)$.

\begin{lemma}
\label{lem:3.1}
Let $T$ denote the tautological bundle on $Y=\mathcal{M}_\vartheta$. Then
\begin{enumerate}
\item[\one] there is an isomorphism of $\kk$-algebras $\phi^\vee\colon A^{\mathrm{op}} \to \End_{\mathscr{O}_Y}(T^\vee);$
\item[\two] the dual bundle $T^\vee \cong\bigoplus_{i \in \qpol_0}L_i^{-1}$ is a tilting bundle on $Y$; and
\item[\three] the induced equivalence of derived categories
\[
\mathbf{R}\!\Hom_{\mathscr{O}_Y}(T^\vee,-)\colon D^b(\coh(Y)) \longrightarrow D^b(\modA)
\]
sends $L_i^{-1}$ to the indecomposable projective $A$-module $Ae_i$ for each $i\in \qpol_0$.
\end{enumerate}
\end{lemma}

\begin{proof}
Reversing the orientation of every arrow in $\qpol$ produces a consistent dimer model $\qpol^{\op}$ with Jacobian algebra $A^{\op}$. To make the distinction between $\qpol$ and $\qpol^{\mathrm{op}}$ explicit, we temporarily write $\Mcal_{\vartheta}(\qpol)$ for $\Mcal_{\vartheta}$. Since the stability parameter $-\vartheta$ is generic for $\qpol^{\mathrm{op}}$, we may form the fine moduli space $\Mcal_{-\vartheta}(\qpol^{\mathrm{op}})$ of $-\vartheta$-stable $A^{\mathrm{op}}$-modules with dimension vector $\underline{1}$. We claim that the underlying varieties $Y=\Mcal_{\vartheta}(\qpol)$ and $\Mcal_{-\vartheta}(\qpol^{\mathrm{op}})$ are isomorphic. To see this, note first the set of perfect matchings for $\qpol$ coincides precisely with those for $\qpol^{\op}$ because the opposite dimer model is obtained by reversing the orientation of each arrow. Moreover, the set of $\vartheta$-stable perfect matchings for $\qpol$ coincides precisely with the $-\vartheta$-stable perfect matchings for $\qpol^{\op}$. Lemma~\ref{lem:fan} implies that the set of stable perfect matchings determines the fan of each moduli space, so the fan of $Y=\Mcal_{\vartheta}(\qpol)$ coincides with that of $\Mcal_{-\vartheta}(\qpol^{\mathrm{op}})$. This proves the claim.

We now establish the statements from the lemma. Recall from Remark~\ref{rem:ishiiueda09}(1) that $H^0(L_j)$ is isomorphic to the space $e_j A e_0$ of classes of paths in $\qpol$ from 0 to $j$, so the space of sections of the $j$th tautological bundle on $\Mcal_{-\vartheta}(\qpol^{\mathrm{op}})$ is isomorphic to the space $e_jA^{\op}e_0$. The zero vertex defines the trivial bundle on every moduli space, so the passage from $Y=\Mcal_{\vartheta}(\qpol)$ to $\Mcal_{-\vartheta}(\qpol^{\mathrm{op}})$ amounts to replacing each tautological line bundle $L_i$ by its inverse $L_i^{-1}$. Thus, the tautological bundle on $\Mcal_{-\vartheta}(\qpol^{\mathrm{op}})$ is $T^\vee \cong \bigoplus_{i \in \qpol_0}L_i^{-1}$. Applying Theorem~\ref{thm:ishiiueda09} to $\Mcal_{-\vartheta}(\qpol^{\mathrm{op}})$ establishes \one\ and \two, where we write the tautological isomorphism as $\phi^\vee\colon A^{\mathrm{op}} \to \End_{\mathscr{O}_Y}(T^\vee)$. For \three, the fact that $T^\vee$ is tilting implies that $\Ext^k_{\mathscr{O}_Y}(L_j^{-1},L_i^{-1})=0$ for all $k>0$ and $i, j\in \qpol_0$, giving
$$
\mathbf{R}\!\Hom_{\mathscr{O}_Y}(T^\vee,L_i^{-1}) \cong \bigoplus_{j \in Q_0}\Hom_{\mathscr{O}_Y}(L_j^{-1},L_i^{-1})\cong \Hom_{\mathscr{O}_Y}(L_i, T),
$$
which is $\End_{\mathscr{O}_{Y}}(T)e_i$ by Remark~\ref{rem:ishiiueda09}(1). The isomorphism from Theorem~\ref{thm:ishiiueda09}\one\ identifies this with the indecomposable projective $A$-module $A e_i$. It remains to note that the codomain of the functor is $D^b(\modA)$ because $(A^{\mathrm{op}})^{\mathrm{op}}=A$. 
 \end{proof}

Example~\ref{exa:McKay} implies that the fine moduli space $Y=\mathcal{M}_\vartheta$ coincides with the $G$-Hilbert scheme when $A=\kk[x,y,z]*G$ is the skew group algebra. The equivalence between $\modA$ and the category of $G$-equivariant coherent sheaves on $\mathbb{A}^3_\kk$ identifies the indecomposable projective module $Ae_{\rho}$ associated to $\rho\in \qpol_0$ with the $G$-equivariant sheaf $\mathscr{O}_{\mathbb{A}^3_\kk}\otimes\rho$, and the Fourier--Mukai transform $\Phi$ induced by the universal family on the $G$-Hilbert scheme studied by Bridgeland--King--Reid~\cite{BKR01} satisfies $\Phi(L_\rho^{-1})\cong \mathscr{O}_{\mathbb{A}^3_\kk}\otimes \rho$ for every $\rho\in \qpol_0$, see Craw--Ishii~\cite[Section~2.4]{CrawIshii04}. Therefore, when working with the moduli space $Y=\mathcal{M}_\vartheta$, we can extend existing results on the McKay correspondence only when we consider the derived equivalence
\[
\Phi(-):= \mathbf{R}\!\Hom_{\mathscr{O}_Y}(T^\vee,-)\colon D^b(\coh(Y)) \longrightarrow D^b(\modA)
\]
 induced by the dual $T^\vee$ of the tautological bundle. We choose to write the quasi-inverse as 
\begin{equation}
\label{eqn:Psi}
\Psi(-):= T^\vee\ltensor_{A} - \;\;\colon D^b(\modA) \longrightarrow D^b(\coh(Y))
\end{equation}
rather than $-\ltensor_{A^{\textrm{op}}} T^\vee$. Lemma~\ref{lem:3.1}\three\ shows that the derived equivalence $\Psi$ sends the projective $A$-module $Ae_i$ to the line bundle $L_i^{-1}$ for each $i\in \qpol_0$. 

\subsection{Images of vertex simples}
Let $Y=\mathcal{M}_\vartheta$ be the fine moduli space defined by any stability parameter $\vartheta$ of the form \eqref{eqn:vartheta} and let $\Psi$ denote the equivalence of derived categories from \eqref{eqn:Psi}. 

We now describe the images under $\Psi$ of the vertex simple $A$-modules $S_i:=\kk e_i$ for $i\in \qpol_0$ in terms of the tautological line bundles on $Y$. Applying the result of Mozgovoy--Reineke~\cite[Proposition~6.2]{MozgovoyReineke10} to the consistent dimer model $\qpol^{\mathrm{op}}$ introduced in the proof of Lemma~\ref{lem:3.1} above shows that the minimal projective $A^{\mathrm{op}}$-module resolution of $S_i$ is
\[
A e_i \xlongrightarrow{\cdot b} \bigoplus_{b\colon h\to i}Ae_{h}\xlongrightarrow{\cdot \omega_{ba}} \bigoplus_{a\colon i\to j} Ae_j \xlongrightarrow{\cdot a} Ae_i  \longrightarrow S_i
\]
where $\omega_{ba} = (-1)^{bpa} p$ if the path $bpa$ is a term in $W$, and it is zero otherwise.

\begin{lemma}
\label{lem:imageSi}
For $i\in \qpol_0$, the object $\Psi(S_i)$ is quasi-isomorphic to the complex
\[
 L_i^{-1} \xlongrightarrow{\phi^\vee(b^{\textrm{op}})} \bigoplus_{b\colon h\to i} L^{-1}_{h} \xlongrightarrow{\phi^\vee(\omega_{ba}^{\textrm{op}})}\bigoplus_{a\colon i\to j} L^{-1}_{j}\xlongrightarrow{\phi^\vee(a^{\textrm{op}})} \underline{L^{-1}_i}
\]
where the term in degree zero is underlined.
\end{lemma}
\begin{proof}
We compute that 
\[
\Psi(S_i) =  T^\vee\ltensor_{A} S_i \cong T^\vee\otimes_{A}  \Bigg(A e_i \xlongrightarrow{\cdot b} \bigoplus_{b\colon h\to i}Ae_{h}\xlongrightarrow{\cdot \omega_{ba}} \bigoplus_{a\colon i\to j} Ae_j \xlongrightarrow{\cdot a} \underline{Ae_i}\Bigg).
\]
This is the given complex because $T^\vee=\bigoplus_{i\in \qpol_0} L_i^{-1}$, and because premultiplication by $a\colon i\to j$ is postmultiplication by $a^{\textrm{op}}\colon j\to i$ which, under the isomorphism $\phi^\vee\colon A^{\mathrm{op}} \to \End_{\mathscr{O}_Y}(T^\vee)$ from Lemma~\ref{lem:3.1}, corresponds to postcomposition with the map $\phi^\vee(a^{\textrm{op}}) \colon L^{-1}_{j}\to L^{-1}_i$.
 \end{proof}
 
 The cohomology sheaves of the complex $\Psi(S_i)$ from Lemma~\ref{lem:imageSi} can be computed using the main result from Craw--Quintero-V\'{e}lez~\cite{CrawQuinteroVelez12b}. To prove Theorem~\ref{thm:intromain2} we must compute $H^{k}(\Psi(S_i))$ for all $k\in \ZZ$. However, the description of $H^{-1}(\Psi(S_i))$ is rather complicated and our proof of Theorem~\ref{thm:intromain1} bypasses this calculation, so for now we describe only $H^{k}(\Psi(S_i))$ for $k\neq -1$ and we defer the case $k=-1$ to Proposition~\ref{prop:H-1}. Recall that the simple module $S_i$ lies in the socle of an $A$-module $V_y$ if and only if $S_i$ is a submodule of $V_y$, in which case we write $S_i\in \soc(V_y)$.
 
\begin{proposition}
\label{prop:cohomology}
For $i \in \qpol_0$, the cohomology sheaves $H^{k}(\Psi(S_i))$ vanish for all $k\not\in \{-2,-1,0\}$. In addition:
\begin{enumerate}
\item[\one] for $k=0$ and for the locus $Z_i = \big\{y\in Y \mid S_i \subseteq \soc(V_y)\big\}$, we have 
\[
H^0(\Psi(S_i)) \cong L_i^{-1}\vert_{Z_i}.
\]
 \item[\two] for $k=-2$ and for $D_i$ the greatest common divisor of the zero loci of the maps $\phi^\vee(b^{\textrm{op}})$ for arrows $b\colon h\to i$ in $\qpol$, we have 
\[
H^{-2}(\Psi(S_i))\cong L_i^{-1}(D_i)\vert_{D_i}.
\]
In fact, $H^{-2}(\Psi(S_i))\cong 0$ for $i\neq 0$.
\end{enumerate}
\end{proposition}
\begin{proof}
The complex from Lemma~\ref{lem:imageSi} satisfies $H^{k}(\Psi(S_i))=0$ for $k\not\in \{-3,-2,-1,0\}$. The arrow $b^{\textrm{op}}\in \qpol^{\textrm{op}}$ defines a nonzero element in the path algebra $A^{\textrm{op}}$, so it's image under the isomorphism $\phi^\vee\colon A^{\mathrm{op}} \to \End_{\mathscr{O}_Y}(T^\vee)$ is a map of line bundles $\phi^\vee(b^{\textrm{op}})$ that is not the zero map. It follows that the left-hand map in the complex from Lemma~\ref{lem:imageSi}, so $H^{-3}(\Psi(S_i))=0$.  For \one, the right-hand map in the complex is surjective everywhere except on the common zero locus $Z_i^\prime$ of the maps $\phi^\vee(a)$ for $a\colon i\to j$, in which case the cokernel is the fibre of $L_i^{-1}$ over the given point. We claim that $Z_i = Z_i^\prime$. Indeed, the simple module $S_i$ is a submodule of a $\vartheta$-stable $A$-module $V_y$ for $y\in Y$ if and only each morphism $\phi(a)\colon L_i\to L_j$ for $a\colon i\to j$ vanishes at $y\in Y$ which holds if and only if each $\phi^\vee(a^{\textrm{op}})\colon L_j^{-1}\to L_i^{-1}$ for $a\colon i\to j$ vanishes at $y\in Y$. This proves the claim and hence proves \one.  For part \two, the complex from Lemma~\ref{lem:imageSi} is a `wheel' of line bundles in the sense of Craw--Quintero-V\'{e}lez~\cite{CrawQuinteroVelez12b}, so the description of $H^{-2}(\Psi(S_i))$ follows  from \cite[Theorem~1.1(3)]{CrawQuinteroVelez12b}. For the final statement, suppose $H^{-2}(\Psi(S_i))\not\cong 0$ for $i\neq 0$. Part \three\ implies that the zero loci of the maps $\phi^\vee(b)$ for $b\colon h\to i$ is nonempty. Thus, there exists $y\in Y$ for which the corresponding $\vartheta$-stable $A$-module $V_y$ has $e_i$ as one of it's generators. This forces $\vartheta_i<0$ for $i\neq 0$, a contradiction.
\end{proof}

\begin{remarks}
\begin{enumerate}
\item[\one] For a finite subgroup $G\subset \SL(2,\CC)$, Crawley--Boevey~\cite[Theorem~2]{Crawley-Boevey00} characterises the exceptional curves in the minimal resolution of $\CC^2/G$ in terms of the loci parametrising modules that contain a simple module in their socle. Proposition~\ref{prop:cohomology}\one\ (and the stronger version Proposition~\ref{prop:type0or1wall} to follow) should be seen as an extension of this result to dimension three.  
\item[\two] The final statement in Proposition~\ref{prop:cohomology} explains our choice of the special stabililty parameter $\vartheta\in \Theta$ and hence our choice of the moduli space $Y=\mathcal{M}_\vartheta$.
\end{enumerate}
\end{remarks}

\subsection{The zero vertex case}
In this section we compute the cohomology sheaves of the object $\Psi(S_0)$ arising from the zero vertex. We show that the derived dual $\Psi(S_0)^\vee = \mathbf{R}\mathcal{H}om(\Psi(S_0),\mathscr{O}_Y)$ is a shift of the structure sheaf of the locus $\tau^{-1}(x_0)$ in $Y$. This implies that $\Psi(S_0)$ is the dualising complex of the locus $\tau^{-1}(x_0)$.

Using Lemma~\ref{lem:widehat}, the complex $\Psi(S_i)$ from Proposition~\ref{lem:imageSi} satisfies
\[
\Gamma\big(U_\sigma,\Psi(S_i)\big) = \left(e_0 \widehat A_\sigma e_i \xlongrightarrow{\cdot b} \bigoplus_{b\colon h\to i} e_0\widehat A_\sigma e_{h}\xlongrightarrow{\cdot \omega_{ba}} \bigoplus_{a\colon i\to j} e_0 \widehat A_\sigma e_j \xlongrightarrow{\cdot a} \underline{e_0 \widehat A_\sigma e_i}\right),
\]
and the dual complex $\Psi(S_i)^\vee$ then satisfies
\[
 \Gamma\big(U_\sigma,\Psi(S_i)^\vee\big) = \left( e_i\widehat A_\sigma e_0  \xlongleftarrow{b\cdot } \bigoplus_{b\colon h\to i} e_{h}\widehat A_\sigma e_0 \xlongleftarrow{\omega_{ba}\cdot } \bigoplus_{a\colon i\to j} e_j\widehat A_\sigma e_0 \xlongleftarrow{a\cdot } \underline{ e_i\widehat A_\sigma e_0 }\right).
\]
We now show that for $i=0$ and for the stability condition $\vartheta$, the object $\Psi(S_i)^\vee$ is a pure sheaf.

\begin{proposition}
\label{prop:dualPsiS0}
The derived dual object $\Psi(S_0)^\vee$ is quasi-isomorphic to the pushforward of the structure sheaf of $\tau^{-1}(x_0)$ shifted by $3$. 
\end{proposition}
\begin{proof}
Let $\sigma\subseteq \Sigma$ be a cone defining the toric chart $U_\sigma$. The ring of sections of $\mathscr{O}_{\tau^{-1}(x_0)}$ restricted
to $U_\sigma$ is the quotient of $\Gamma(U_\sigma,\mathscr{O}_Y)=e_0\widehat A_\sigma e_0$ by the sections that are zero on ${\tau^{-1}(x_0)}\cap U_\sigma$. The quotient is path graded, so we can write this quotient as a direct sum of one-dimensional vector spaces, one for each weak path in $e_0\widehat A_\sigma e_0$ that is nonzero for some representation in $\tau^{-1}(x_0)$:
\[
\Gamma(U_\sigma, \mathscr{O}_{\tau^{-1}(x_0)}) 
= \bigoplus_{
\substack{p  \in e_0 A_\sigma e_0\\
\exists V \in  \tau^{-1}(x_0)\cap U_\sigma \mid v_p\neq 0}} 
\kk p.
\]
On the other hand $\Gamma(U_\sigma, \Psi(S_0)^\vee)$ also decomposes as a direct sum according to path degree: for each weak path $p\in e_0\widehat{A}e_0$ we have 
$$
\Gamma(U_\sigma,\Psi(S_0)^\vee)_p := \Bigg(  
(e_0\widehat A_\sigma e_0)_{\omega p}  \xlongleftarrow{b \cdot } \bigoplus_{b\colon h\to 0} (e_{h}\widehat A_\sigma e_0)_{\omega_bp} \xlongleftarrow{\omega_{ba} \cdot } \bigoplus_{a\colon 0\to j} (e_j\widehat A_\sigma e_0)_{ap} \xlongleftarrow{a \cdot } \underline{(e_0\widehat A_\sigma e_0)_p} \Bigg).
$$
Here $\omega p$ is the path obtained by adjoining a boundary cycle to $p$ 
and $\omega_b=\omega b^{-1}$. Each of the terms $(e_i \widehat A_\sigma e_j)_r$ is a $\kk$-vector space of dimension one or zero, depending on whether or not $r\in e_i \widehat A_\sigma e_j$.  Now, $p \in \widehat A_\sigma$ implies that $\omega p \in \widehat A_\sigma$, so we have that the whole complex disappears if $\omega p \notin \widehat A_\sigma$. We may therefore assume that $\omega p \in \widehat A_\sigma$.

We are now going to show that if $\omega p$ is nonzero for some representation $V \in \tau^{-1}(x_0)\cap U_\sigma$, then the complex $\Gamma(U_\sigma,\Psi(S_0)^\vee)_p$ is concentrated in degree $3$ and consists only of $\kk \omega p[3]$. If $\omega p$ is zero on $\tau^{-1}(x_0)\cap U_\sigma$ we will show that the homology of the complex is zero. From this we can then conclude that
\[
\Gamma(U_\sigma,\Psi(S_0)^\vee)\cong 
\bigoplus_{\substack{\omega p  \in e_0 A_\sigma e_0\\
\exists V \in  \tau^{-1}(x_0)\cap U_\sigma \mid v_{\omega p}\neq 0}} \kk \omega p[3].
\]
Renaming $\omega p$ as $p$, we see that this is equal to
$\Gamma(U_\sigma, \mathscr{O}_{\tau^{-1}(x_0)})[3]$ and we are done.

Suppose that $\omega p$ is nonzero for $V \in \tau^{-1}(x_0)\cap U_\sigma$. Now, $V$ sits in some torus orbit $Y_\varsigma$ with $\varsigma \subset \sigma$. By Remark~\ref{rem:cosupp}(3) we have that $\deg_\rho \omega p = 0$ for all $\rho \subset \varsigma$. Combining Lemma~\ref{lem:specialtheta} with Remark~\ref{rem:cosupp}(3) we get that for every arrow $b\colon h\to 0$ we can find a $\rho \subset \varsigma$ such that $\deg_\rho \omega_b p= \deg_\rho \omega p - \deg_\rho b=0-1<0$. Therefore all $(e_{h}\widehat A_\sigma e_0)_{\omega_bp}=0$ and the only nonzero term in 
$\Gamma(U_\sigma,\Psi(S_0)^\vee)_p$ is $(e_0\widehat A_\sigma e_0)_{\omega p}$.

If $\omega p$ is zero on $\tau^{-1}(x_0)\cap U_\sigma$ then 
we define a cone $\varsigma$  which is spanned by the 
$\rho \subset \sigma$ for which $\deg_\rho \omega p=0$. 
Now, $\omega_p$ is nonzero on the torus orbit of this cone, so $Y_\varsigma \not\subset \tau^{-1}(x_0)$.

If $\varsigma=0$ then $\deg_\rho p\ge 0$ for all $\rho\subset \sigma$ and
$p \in \widehat A_\sigma$. All the terms in the complex $\Gamma(U_\sigma,\Psi(S_0)^\vee)$ remain, so the complex has zero homology.

If $\varsigma\ne 0$ then by Lemma~\ref{lem:specialtheta} there must be at least one arrow arriving in the vertex $0$ which is nonzero on $Y_\varsigma$. If two nonzero arrows $b_0,b_1$ have head at 0, then the stability of the representations in $Y_\varsigma $ implies 
we can find nonzero paths $p_0,p_1$ such that $b_0p_0$ and $b_1p_1$ are nonzero cycles through $0$. The cycles
$b_0p_0b_1p_1$ and $b_1p_1b_0p_0$ are equivalent in $A$ (because $e_0Ae_0\cong Z(A)$ is commutative), so there is a sequence of relations that connects one to the other. This means that there is a sequence of equivalent paths $b_0p_0b_1p_1=q_0,\dots, q_u=b_1p_1b_0p_0$ such that each one is constructed
from the previous by flipping over a boundary cycle. 
All arrows $b_i$ inside the area bounded by the lifts of $b_0p_0b_1p_1$ and $b_1p_1b_0p_0$ in the universal cover must occur at 
the end of some of these paths $q_l=rb_i$, so these $b_i$ must be nonzero (see Figure~\ref{fig:PathsPictures} on the left). 
 Similarly the $\omega_{ab}$ inside this area must be nonzero because they must occur in $r$'s, otherwise we could not apply any relations $\langle \partial_a W \mid a\in \qpol_1\rangle$. However, if $\varsigma$ is nonzero, then not all $b_i$ can be nonzero otherwise all $\omega_{ab}$ would be nonzero and yet all $a$ must be zero, contradicting stability of the representations in $Y_{\varsigma}$. It follows that the $b_i$ that are nonzero on $Y_{\varsigma}$ form a segment and those are the terms that do not disappear in the third column. The terms that do not disappear in the second column are those that are behind two nonzero arrows $b_i, b_{i+1}$. The first column always disappears.  The connected subdiagram is as shown in Figure~\ref{fig:PathsPictures} on the right. 
\begin{figure}[ht!]
\centering
\[
 \xymatrix@=1.3cm{
\vtx{}\ar@{.>}[rr]^{p_0}&&\vtx{}\ar@{->}[r]^{b_0}&\vtx{}\\
\vtx{}\ar@{->}[u]^{b_1}&&\vtx{}\ar@{->}[ur]^{b_i}&\vtx{}\ar@{->}[u]^{b_1}\ar@{<.}[dd]_{p_1}\\
 &&&\\
\vtx{}\ar@{.>}[rr]^{p_0}\ar@{.>}[ruru]^{r}\ar@{.>}[uu]^{p_1}&&\vtx{}\ar@{->}[r]^{b_0}&\vtx{}}
\hspace{2cm}
\xymatrix@=.1cm{
 &&\cL{e_{\tail(b_i)}}{e_0}\ar[rrddd]&&\\
 \cL{e_{\head(a_i)}}{e_0}\ar[rru]\ar[rrd]&&&&\\
 &&\cL{e_{\tail(b_{i+1})}}{e_0}\ar[rrd]&&\\
\vdots \ar[rru]\ar[rrd]&&\vdots && \cL{e_0}{e_0}\\ 
 &&\cL{e_{\tail(b_{j-1})}}{e_0}\ar[rru]&&\\
\cL{e_{\head(a_{j-1})}}{e_0}\ar[rru]\ar[rrd]&&&&\\
 &&\cL{e_{\tail(b_j)}}{e_0}\ar[rruuu]&&
 }
 \]
\caption{A connected subdiagram in the quiver $\qpol$}
\label{fig:PathsPictures}
\end{figure}
It follows that the complex has trivial homology as claimed.
\end{proof}

To state the main result of this section, let $S$ be any connected scheme of finite type over $\kk$ with structure morphism $f_S\colon S\to \Spec(\kk)$. Assume further that $S$ is Cohen--Macaulay. Recall from \cite{Hartshorne} that the dualising sheaf of $S$ is $\gbf{\omega}_S:= f_S^{!}(\mathscr{O}_{\Spec(\kk)})$. 
If $S$ is smooth and equidimensional, we have $\gbf{\omega}_S = \omega_S[\dim S]$ where $\omega_S$ is the canonical line bundle of $S$. Given a closed immersion $\iota_F\colon F\hookrightarrow S$, the structure morphism of $F$ is simply $f_Y\circ \iota_F$ and 
hence $\gbf{\omega}_F = {\iota_F}^{!} \gbf{\omega}_S$.

\begin{proposition}
\label{prop:PsiS0}
The object $\Psi(S_0)$ is quasi-isomorphic to the dualizing complex of $\tau^{-1}(x_0)$. 
\end{proposition}
\begin{proof}
Let $F=\tau^{-1}(x_0)$ denote the fibre over the torus-invariant point of $X$. Taking the derived dual of the quasi-isomorphism from Proposition~\ref{prop:dualPsiS0} gives 
 \[
 \Psi(S_0) \cong \mathbf{R}\mathcal{H}om({\iota_{F}}_* \mathscr{O}_F,\mathscr{O}_Y)[3].
 \]
Grothendieck duality for the proper morphism $\iota_F\colon F\to Y$ \cite[VII, Corollary 3.4(c)]{Hartshorne} gives 
\[ 
\Psi(S_0) \cong {\iota_{F}}_* \mathbf{R}\mathcal{H}om(\mathscr{O}_F,{\iota_F}^{!}\mathscr{O}_Y)[3]\cong  {\iota_{F}}_*{\iota_F}^{!}\mathscr{O}_Y[3]\cong  {\iota_{F}}_*{\iota_F}^{!}\gbf{\omega}_Y\cong  {\iota_{F}}_*\gbf{\omega}_F,
\]
where smoothness of $Y$ and crepancy of $\tau$ give $\gbf{\omega}_Y\cong \mathscr{O}_Y[3]$. 
\end{proof}

\begin{corollary}
The object $\Psi(S_0)$ is a pure sheaf if and only if $\tau^{-1}(x_0)$ is equidimensional.
\end{corollary}
\begin{proof}
For $d\in \{1,2\}$, let $F_d$ denote the subscheme of $F=\tau^{-1}(x_0)$ defined by the union of the $d$-dimensional irreducible components. Following the approach of \cite[Proposition~5.6]{CautisCrawLogvinenko12} verbatim, we deduce from Proposition~\ref{prop:PsiS0} that:
\begin{enumerate}
\item[\one] $H^{-2}(\Psi(S_0)) = \mathscr{O}_{F_2}\otimes \mathscr{O}_Y(F_2)$;
\item[\two] $H^{-1}(\Psi(S_0)) = \omega_{F_1}(F_2)$;
\item[\three] $H^{-i}(\Psi(S_0)) = 0$ for $i\neq 1,2$.
\end{enumerate}
This gives $\supp(H^{-i}(\Psi(S_0))) = \supp(F_i)$ for $i=1,2$ and $\supp(H^{-i}(\Psi(S_0))) = \emptyset$ otherwise.
\end{proof}

\section{The main result for nonzero vertices}
In this section we compute explicitly the object $\Psi(S_i)$ for any nonzero vertex $i\in \qpol$ whenever it has nonvanishing cohomology in degree zero. The key observation is that the support of the cohomology sheaf $H^0(\Psi(S_i))$ coincides with the unstable locus for an especially simple type of wall of the GIT chamber defining the moduli space $Y=\mathcal{M}_\vartheta$. This enables us to complete the proof of Theorem~\ref{thm:intromain1}.

\subsection{Stability conditions and the Grothendieck group}
Identify the space of dimension vectors $\ZZ_{\qpol_0}$ with the free abelian group $\bigoplus_{i\in \qpol_0} \mathbb{Z}S_i$ generated by the vertex simple $A$-modules. The vector space of stability conditions on the quiver $\qpol$ becomes
\[
\Theta:=
\Bigg\{\theta\colon \bigoplus_{i\in \qpol_0} \mathbb{Z}S_i\to \QQ \,\bigg|\, \textstyle{\sum_{i\in \qpol_0} \theta(S_i)} = 0\Bigg\}.
\]
For $i\in \qpol_0$, it is convenient to introduce the hyperplane 
\[
S_i^\perp:= \{\theta\in \Theta \mid \theta(S_i)=0\}.
\]

We now express $\Theta$ in terms of the  Grothendieck group $K(\modA)$ of the abelian category $\modA$. Consider the bilinear form $\chi\colon K(\modA)\times  \bigoplus_{i\in \qpol_0} \mathbb{Z}S_i\longrightarrow \ZZ$ defined by setting 
\begin{equation}
\label{eqn:pairing}
\chi(\alpha,\beta) = \sum_{k\in \ZZ} (-1)^k\dim \Ext_A^k(\alpha,\beta)
\end{equation}
for $\alpha\in K(\modA)$ and $\beta\in \bigoplus_{i\in \qpol_0} \mathbb{Z}S_i$. 
Since $A$ has finite global dimension, the isomorphism classes of the indecomposable projective modules $\{[P_i] \mid i\in \qpol_0\}$ generate $K(\modA)$. Moreover, for  $i, j\in \qpol_0$ and $k\in \ZZ$, we have $\Ext^k_A(P_i,S_j)=\kk$ for $i=j$ and $k=0$, and $\Ext^k_A(P_i,S_j) = 0$ otherwise. It follows that \eqref{eqn:pairing} is a perfect pairing. In particular, $K(\modA)$ is finitely generated and free with basis $\{[P_i] \mid i\in \qpol_0\}$ and the space of stability conditions is
\[
\Theta=\Big\{\textstyle{\sum_{i\in \qpol_0} \theta_i[P_i]\in K(\modA)\otimes_\ZZ \QQ \,\big|\, \sum_{i\in \qpol_0} \theta_i=0}\Big\}.
\]
\begin{remark}
With this description of stability conditions, we have that:
\begin{enumerate}
\item $\theta=\sum_{j\in \qpol_0} \theta_j[P_j]\in \Theta$ lies in the hyperplane $S_i^\perp$ for $i\in \qpol_0$ if and only if $\theta_i=0$.
\item the special stability condition defining $Y=\mathcal{M}_\vartheta$ is simply 
\[
\vartheta = \big(1-\vert\qpol_0\vert\big) [P_0] + \sum_{i>0} [P_i].
\]
\end{enumerate}
\end{remark}

\subsection{The wall and chamber structure}
We now recall the consequences for the birational geometry of the fine moduli space $Y=\mathcal{M}_\vartheta$ as we vary the stability parameter in $\Theta$. Recall that a stability parameter $\theta\in \Theta$ is said to be generic if every $\theta$-semistable $A$-module of dimension vector $\underline{1}=\sum_{i\in \qpol_0} S_i$ is $\theta$-stable. As is standard in GIT, the chamber $C\subset \Theta$ containing a generic parameter $\theta\in \Theta$ is defined to be the set of all generic $\theta^\prime\in\Theta$ such that every $\theta$-stable $A$-module of dimension vector $\underline{1}$ is $\theta^\prime$-stable and vice-versa. 

The following result is well known. In fact it differs from the statement of \cite[Lemma~3.1]{CrawIshii04} or \cite[Remark~6.3]{IshiiUeda13} only in the observation that the chambers can be described purely in terms of torus-invariant $A$-modules; this observation was well known to the authors of those papers.

\begin{lemma}
\label{lem:ThetaFan}
The space of stability parameters $\Theta$ supports a polyhedral fan such that the GIT chambers are precisely the interiors of all top-dimensional cones in the fan. Moreover, every chamber is of the form
\[
C= \left\{\theta^\prime\in \Theta \mid \begin{array}{cc} \theta^\prime(S)>0 \text{ for every nontrivial submodule }S \text{ of every}\\ \text{ torus-invariant }\theta\text{-stable }A\text{-module of dimension vector }\underline{1}\end{array}\right\}
\]
for some generic parameter $\theta\in \Theta$.
\end{lemma}
\begin{proof}
For generic $\theta\in \Theta$, the toric variety $\mathcal{M}_\theta$ has finitely many torus-invariant points, so there are finitely many torus-invariant $\theta$-stable $A$-modules of dimension vector $\underline{1}$, each of which has only finitely many submodules. Thus, the nonempty open cone $C$ is cut out by finitely many strict linear inequalities, so its closure $\overline{C}\subset \Theta$ is a rational polyhedral cone of full dimension. 

To prove that $C$ is the GIT chamber containing $\theta$, we proceed in three steps. First we prove that every $\theta$-stable $A$-module is $\theta^\prime$-stable for $\theta^\prime\in C$. Equivalently, we must show that $C$ equals
\[
C^\prime= \left\{\theta^\prime\in \Theta \mid \begin{array}{cc} \theta^\prime(S)>0 \text{ for every nontrivial submodule }S \text{ of} \\ \text{every }\theta\text{-stable }A\text{-module of dimension vector }\underline{1}\end{array}\right\}.
\]
One inclusion is obvious. For the other let $\theta^\prime\in C$ and let $S\subseteq V_y$ denote a submodule of the $\theta$-stable $A$-module corresponding to a point $y\in Y$. The torus-action induces an isomorphism between $V_y$ and the $\theta$-stable $A$-module $V_{y^\prime}$ defined by the distinguished point $y^\prime$ in the torus-orbit containing $y$, so $V_{y^\prime}$ has a submodule $S^\prime$ that's isomorphic to $S$. Since the torus-invariant $\theta$-stable $A$-module $V_0$ for any toric chart containing this torus orbit is obtained from $V_{y^\prime}$ by setting certain maps to zero, it follows that any submodule $S^\prime\subseteq V_{y^\prime}$ is isomorphic to a submodule $S_0\subseteq V_0$. In particular, $\theta^\prime(S) = \theta^\prime(S^\prime) = \theta^\prime(S_0) > 0$ because $\theta^\prime\in C$, giving $C\subseteq C^\prime$ and hence $C^\prime = C$. This completes step one, and implies that $\mathcal{M}_{\theta}$ is an open subset of $\mathcal{M}_{\theta^\prime}$ for $\theta^\prime\in C$.  Our second step is to deduce that $\mathcal{M}_{\theta} = \overline{\mathcal{M}_{\theta^\prime}}$ for $\theta^\prime\in C$, and for this we merely sketch the argument following \cite[Lemma~3.1]{CrawIshii04} or \cite[Remark~6.3]{IshiiUeda13}. Suppose there exists $y\in \overline{\mathcal{M}_{\theta^\prime}}$ such that a representative $V_y$ of the corresponding S-equivalence class of $\theta^\prime$-semistable $A$-modules is not $\theta$-stable. Since $\theta$ is generic by assumption, we may use the derived equivalence $\Phi_\theta$ induced by the tautological bundle on $\mathcal{M}_{\theta}$. Applying the argument of \cite[\S8]{BKR01} to any object $E\in D^b(\coh(\mathcal{M}_\theta))$ satisfying $\Phi_\theta(E) = V_y$ leads to a contradiction, giving $\mathcal{M}_{\theta} = \overline{\mathcal{M}_{\theta^\prime}}$ as claimed. The third step is to observe that the equality of step two implies that $\mathcal{M}_{\theta^\prime} = \overline{\mathcal{M}_{\theta^\prime}}$, so every $\theta^\prime\in C$ is generic and, moreover, that $\mathcal{M}_{\theta^\prime} = \mathcal{M}_{\theta}$, so every $\theta^\prime$-stable $A$-module is $\theta$-stable for $\theta^\prime\in C$. In other words,  $C$ is indeed the GIT chamber containing $\theta$ as required.

It remains to prove that the faces of all possible cones of the form $\overline{C}\subset \Theta$ for some generic $\theta$ define a fan in $\Theta$. For generic $\theta$, the moduli space $\mathcal{M}_\theta$ is isomorphic to its coherent component $Y_\theta$ (see Remark~\ref{rem:ishiiueda08}), so the chamber $C$ containing $\theta$ coincides with one of the GIT chambers for the torus-action on the affine toric variety that defines $Y_\theta$ \cite[Proposition~2.14]{CrawQuinterovelez12}. It is well known that the wall and chamber structure in the set of GIT stability parameters for a torus-action on an affine toric variety determine a fan.
  \end{proof}

\begin{definition}
The \emph{GIT chamber decomposition} for the consistent dimer model $\qpol$ is the fan from Lemma~\ref{lem:ThetaFan}. A \emph{chamber} $C$ in $\Theta$ is the interior of any top-dimensional cone in the fan, and a \emph{wall} of a chamber $C$ is any codimension-one face of the closure $\overline{C}$.
\end{definition}

Let $C, C^\prime\subset \Theta$ be adjacent chambers separated by a wall $W=\overline{C}\cap \overline{C^\prime}$, and choose parameters $\theta\in C$, $\theta^\prime\in C^\prime$ and $\theta_0\in W$ in the relative interior of the wall. Since $\theta_0$ lies in the closure of $C$, the representations $V_y$ for $y\in Y$ are all $\theta_0$-semistable, and the natural projective morphism obtained by variation of GIT quotient 
\[
f\colon \mathcal{M}_\theta\longrightarrow \overline{\mathcal{M}_{\theta_0}}
\]
sends $y\in \mathcal{M}_\theta$ to the S-equivalence class of $V_y$ in the category of $\theta_0$-semistable representations. Explicitly, $f(y)=f(y^\prime)$ for distinct points $y, y^\prime\in Y$ if and only if the corresponding $\vartheta$-stable $A$-modules $V_y, V_{y^\prime}$ are strictly $\theta_0$-semistable, and each contains a $\theta_0$-stable submodule $S\subset V_y$ and $S^\prime\subset V_{y^\prime}$ such that both submodules and quotient modules are isomorphic:
 \begin{equation}
 \label{eqn:destabilisingGeneral}
 \begin{CD}   
    0@>>> S @>>> V_y @>>> V_y/S @>>> 0 \\
     @. @VV{\cong}V @. @VV{\cong}V @. \\
   0@>>> S^\prime @>>> V_{y^\prime} @>>> V_{y^\prime}/S^\prime @>>> 0.
 \end{CD}
\end{equation}

 \begin{definition}
 The \emph{unstable locus} in $\mathcal{M}_\theta$ for the wall $W$ containing $\theta_0$ in its relative interior is the locus in $\mathcal{M}_\theta$ parametrising strictly $\theta_0$-semistable representations.
 \end{definition}

We now describe the classification of walls of a chamber $C$ in terms of the geometry of the morphism $f$. First, since $\mathcal{M}_\theta$ is isomorphic to the coherent component $Y_\theta$ for generic $\theta\in \Theta$ as in Remark~\ref{rem:ishiiueda08}, the proof of \cite[Theorem~3.15]{CrawQuinterovelez12} shows that $f$ factors through the categorical quotient $Y_{\theta_0}$. The same is true for $\mathcal{M}_{\theta^\prime}\cong Y_{\theta^\prime}$, so we obtain a diagram
 \begin{equation}
 \label{eqn:cdFM}
 \xymatrix@C=.6em{
   Y_\theta\ar[drrr]^{f} &&&&&& Y_{\theta^\prime}\ar[dlll]_{f^\prime}\\&&& Y_{\theta_0} }
 \end{equation}
of projective toric morphisms. The following result is due to Ishii--Ueda~\cite{IshiiUeda13}:

\begin{proposition}[Classification of walls]
\label{prop:classification}
Let $C$ be any chamber in $\Theta$ with $\theta\in C$. Let $W$ be a wall of $C$ that contains $\theta_0$ in its relative interior. Then $W$ is either:
\begin{itemize}
\item $\operatorname{Type~0}$: $f\colon Y_\theta\to Y_{\theta_0}$ is an isomorphism onto its image, and the unstable locus for $W$ is a connected union of compact torus-invariant divisors; 
\item $\operatorname{Type~I}$: $f\colon Y_\theta\to Y_{\theta_0}$ contracts a torus-invariant $(-1,-1)$-curve $\ell$ to a point, and the unstable locus for $W$ is the rational curve $\ell$; 
\item $\operatorname{Type~III}$: $f\colon Y_\theta\to Y_{\theta_0}$ contracts a torus-invariant Hirzebruch surface $\mathbb{F}_n$ to a torus-invariant rational curve, and the unstable locus for $W$ is the surface $\mathbb{F}_n$.
\end{itemize}
\end{proposition}

\begin{proof}[Outline of the proof]
The morphism $f$ is determined by the line bundle $L_W:=\bigotimes_{i\in \qpol_0} L_i^{\theta_0(i)}$ which is nef precisely because $\theta_0$ lies in the closure of the chamber $C$. There are two cases. If $L_W$ is ample, then $f$ is an isomorphism onto its image and the wall is of type 0. Otherwise, the assumption that $\theta_0$ lies in the relative interior of $W$ implies that $L_W$ lies in the relative interior of a facet of the nef cone of $Y_\theta$. According to (the toric version of) the classification of facets of the nef cone for a Calabi--Yau threefold by Wilson~\cite{Wilson92}, the morphism determined by any such line bundle either: \one\ contracts a curve to a point; \two\ contracts a surface to a point; or \three\ contracts a surface to a curve. We say that the wall $W$ is of type I, II or III if it induces a contraction of type \one, \two\ or \three\ respectively. Ishii--Ueda~\cite[Lemma~10.5]{IshiiUeda13} show that type II walls do not exist, completing the classification of walls into types 0, I or III as stated. The investigation of the unstable locus requires a much more delicate analysis, see \cite[\S11]{IshiiUeda13}.
\end{proof}

\begin{remark}
The classification of walls is slightly simpler in the McKay quiver case \cite{CrawIshii04} because the unstable locus for a wall of type 0 is necessarily irreducible. Compare Example~\ref{exa:longhex} below.
\end{remark}

\subsection{Nonvanishing cohomology in degree zero}
We can now establish the link between our calculation of $H^0(\Psi(S_i))$ from Proposition~\ref{prop:cohomology}
 with the unstable locus of certain walls of the GIT chamber $C$ containing $\vartheta$.

\begin{proposition}
\label{prop:TFAE}
Let $i\in \qpol_0$ be a nonzero vertex. The following are equivalent:
\begin{enumerate}
\item[\one] the sheaf $H^0(\Psi(S_i))$ is nonzero;
\item[\two] there exists a torus-invariant $\vartheta$-stable $A$-module that contains $S_i$ in its socle;
\item[\three] $\overline{C}\cap S_i^\perp$ is a wall of the chamber $C$ containing $\vartheta$.
\end{enumerate}
 If one and hence any of these conditions is satisfied, then $H^0(\Psi(S_i))\cong L_i^{-1}\vert_{Z_i}$
where $Z_i$ is the unstable locus of the wall $\overline{C}\cap S_i^\perp$.
\end{proposition}
\begin{proof}
Proposition~\ref{prop:cohomology} establishes that $H^0(\Psi(S_i))\cong L_i^{-1}\vert_Z$ where $Z = \{y\in Y \mid S_i\subseteq \soc(V_y)\}$, so $H^0(\Psi(S_i))$ is nonzero if and only if there exists $\vartheta$-stable $A$-module that contains $S_i$ in its socle. The proof of Lemma~\ref{lem:ThetaFan} shows that a $\vartheta$-stable $A$-module contains $S_i$ in its socle if and only if a torus-invariant $\vartheta$-stable $A$-module contains $S_i$ in its socle, so \one\ and \two\ are equivalent.

 To show that \two\ and \three\ are equivalent, suppose that the $\vartheta$-stable $A$-module $V_y$ for $y\in Y$ contains $S_i$ in its socle. Stability implies that $\theta(S_i)> 0$ for all $\theta \in C$, giving $\theta(S_i)\geq 0$ for all $\theta\in \overline{C}$. Then $S_i$ lies in the cone dual to $\overline{C}$ and $\overline{C}\cap S_i^{\perp}$ is a face of $\overline{C}$. Since the top-dimensional cone $\overline{\Theta_+}$ lies in $\overline{C}$ and $\overline{\Theta_+}\cap S_i^{\perp}$ is a codimension-one face of $\overline{\Theta_+}$, it follows that $\overline{C}\cap S_i^{\perp}$ is actually a wall of $C$. Conversely, assume that $\overline{C}\cap S_i^\perp$ is a wall of $C$ for some nonzero vertex $i\in \qpol_0$. The unstable locus for the wall destabilises all $\vartheta$-stable $A$-modules $V_y$ that contain $S_i$ as a submodule, and there must exist at least one such since the unstable locus is nonempty. This establishes the equivalence of \two\ and \three. In fact it also demonstrates that the unstable locus for the wall $\overline{C}\cap S_i^{\perp}$ coincides with the locus $Z_i = \{y\in Y \mid S_i\subseteq \soc(V_y)\}$ from Proposition~\ref{prop:cohomology}. The final statement is now immediate from Proposition~\ref{prop:cohomology}.
\end{proof}

\begin{lemma}
\label{lem:0orI}
Let $i\in \qpol_0$ be nonzero. Every wall of the form $\overline{C}\cap S_i^\perp$ is either of type 0 or I. In particular, the support of $H^0(\Psi(S_i))$ is either a single $(-1,-1)$-curve or a connected union of compact torus-invariant divisors.
\end{lemma}
\begin{proof}
To prove the first statement we need only prove that a wall of the form $\overline{C}\cap S_i^\perp$ for $i\neq 0$ cannot be of type III. We suppose otherwise and seek a contradiction. Let $Z$ denote the unstable locus of the wall. The morphism $f\colon \mathcal{M}_\vartheta\to \overline{\mathcal{M}_{\theta_0}}$ induced by moving the GIT parameter into the wall contracts a surface to a curve. Let $z\in Z$ be a torus-invariant point. We know from \cite[Lemma~10.1]{IshiiUeda13} that the fibre of this morphism over the torus-invariant point $f(z)$ of the base curve is $\ell\cong \mathbb{P}(\Ext^1(V_z/S_i, S_i)^\vee)$ where 
\begin{equation}
\begin{CD}   
    0@>>> S_i @>>> V_z @>>> V_z/S_i @>>> 0
 \end{CD}
\end{equation}
is the $\theta_0$-destabilising sequence for $V_z$. Since $f$ contracts a divisor, \cite[Lemma~10.4]{IshiiUeda13} shows that the dimension of the space $\Ext^1(V_z/S_i, S_i)^\vee$ counts the number of connected components of the boundary of the support of the destabilising submodule. In this case, the destabilising submodule is $S_i$, so $\dim_\kk\Ext^1(V_z/S_i, S_i)^\vee$ counts the number of components of the boundary of the unique tile in the dimer model $\Gamma$ corresponding to vertex $i\in \qpol_0$. But since the fibre of $f$ is a curve, this means that the tile in the \emph{dual} decomposition of the real torus corresponding to vertex $i$ must have boundary with two components. But then the tile is not simply connected, contradicting the definition of a dimer model. It follows that a wall of the form $\overline{C}\cap S_i^\perp$ for $i\neq 0$ is not of type III. The final statement is now immediate from Proposition~\ref{prop:classification}.
\end{proof}
 
\begin{remark}
For the McKay quiver, Craw--Ishii~\cite[Proposition~9.3]{CrawIshii04} prove that the unstable locus of a wall of type 0 of the form $\overline{C}\cap S_i^\perp$ is necessarily irreducible. However, this need not be the case for consistent dimer model algebras, see Example~\ref{exa:longhex}.
\end{remark}

For any basic triangle $\tau$ in $\Sigma$, let $U_\tau\cong \mathbb{A}_\kk^3$ denote the corresponding toric chart and write $V_0$ for the torus-invariant $\vartheta$-stable $A$-module corresponding to the origin in $U_\tau$. For any $j\in \qpol_0$, we compute the unique generator of the restriction of the line bundle $L_j\vert_{U_\tau}$ as follows. Since $V_0$ is $\vartheta$-stable, Lemma~\ref{lem:specialtheta} gives a path $p(j)$ from 0 to $j$ such that the corresponding map in the quiver representation $V_0$ is nonzero. Remark~\ref{rem:ishiiueda09}(1) shows that $H^0(L_j)$ is isomorphic to the space $e_j A e_0$ of classes of paths in $\qpol$ from 0 to $i$, so the path $p(j)$ defines a section of $L_j$ that is nonzero at the origin of $U_\tau$. This section is unique because the $A$-module $V_0$ is torus-invariant.

\begin{lemma}
\label{lem:floppingcurve}
Let $\ell$ be a $(-1,-1)$-curve in $Y$ that arises as the unstable locus for a wall of type I of the form $\overline{C}\cap S_i^{\perp}$ for some nonzero $i\in \qpol_0$. Then $L_j\vert_{\ell}\cong \mathscr{O}_\ell$ for all $j\neq i$ and $L_i\vert_{\ell}\cong \mathscr{O}_\ell(1)$.
\end{lemma}
\begin{proof}
Let $\tau, \tau^\prime$ denote the basic triangles that lie adjacent in the triangulation $\Sigma$ such that the torus-invariant points of $\ell$ sit at the origin in the toric charts $U_{\tau}, U_{\tau^\prime}$. Let $V_0(\tau)$ and $V_0(\tau^\prime)$ denote the torus-invariant $\vartheta$-stable $A$-modules corresponding to the origin in $U_{\tau}$ and $U_{\tau^\prime}$ respectively. Choose $\theta_0$ in the relative interior of the type I wall $\overline{C}\cap S_i^{\perp}$. The morphism $f\colon Y_\vartheta \to Y_{\theta_0}$ induced by variation of GIT quotient contracts $\ell$ to a point, so $f(V_0(\tau))$ and $f(V_0(\tau^\prime))$ are S-equivalent. In particular, the $\theta_0$-destabilising quotient modules for $V_0(\tau)$ and $V_{0}(\tau^\prime)$ are isomorphic:
 \begin{equation}
 \label{eqn:destabilising}
 \begin{CD}   
    0@>>> S_i @>>> V_0(\tau) @>>> V_0(\tau)/S_i @>>> 0 \\
     @. @| @. @VV{\cong}V @. \\
   0@>>> S_i @>>> V_0(\tau^\prime) @>>> V_0(\tau^\prime)/S_i @>>> 0.
 \end{CD}
\end{equation}
These quotient modules encode paths from vertex 0 to any vertex $j\neq i$ in $V_0(\tau)$ and $V_0(\tau^\prime)$. Since they are isomorphic,  the generating sections of $L_j\vert_{U_{\tau}}$ and $L_j\vert_{U_{\tau^\prime}}$ coincide, giving $L_j\vert_{U_\tau}\cong L_j\vert_{U_{\tau^\prime}}$ for $j\neq i$. In particular, $L_j\vert_{\ell}\cong \mathscr{O}_\ell$ for all $j\neq i$. For the remaining case, note that $i$ is not the zero vertex, so Ishii-Ueda~\cite[Corollary~11.22]{IshiiUeda13} shows that $\deg L_j\vert_{\ell}$ is either 0 or 1 for any $j\in \qpol_0$. Since $V_0(\tau)\not\cong V_0(\tau^\prime)$, we have $L_i\vert_{\ell}\not\cong \mathscr{O}_\ell$, hence $L_i\vert_{\ell}\cong \mathscr{O}_\ell(1)$ as required.
\end{proof}

\begin{proposition}
\label{prop:type0or1wall}
Let $i\in \qpol_0$ be a nonzero vertex. If $H^{0}(\Psi(S_i))$ is nonzero then  
\[
\Psi(S_i) \cong L_i^{-1}\vert_{Z_i}
\]
where $Z_i$ is the unstable locus for the wall $\overline{C}\cap S_i^\perp$ of the GIT chamber $C$ containing $\vartheta$.
\end{proposition}
\begin{proof}
Suppose first that the support of $H^{0}(\Psi(S_i))$ contains a divisor. Proposition~\ref{prop:TFAE} implies that $\supp(H^{0}(\Psi(S_i)))$ is a divisor $D_i$ that coincides with the unstable locus for the type 0 wall of the form $\overline{C}\cap S_i^\perp$. If $D_i$ is irreducible then the proof of \cite[Corollary~5.6]{CrawIshii04} can be adapted from the McKay quiver case to any consistent dimer model to give $\Psi(S_i) \cong L_i^{-1}\vert_{D_i}$. More generally, even if $D_i$ is reducible, Ishii-Ueda~\cite[Corollary~11.15]{IshiiUeda13} prove that $\Psi(S_i) \cong L_i^{-1}\vert_{D_i}$ as required. 

Otherwise, the support of $H^{0}(\Psi(S_i))$ is a $(-1,-1)$-curve $\ell_i$ that coincides with the unstable locus for the type I wall $\overline{C}\cap S_i^\perp$. In this case, we compute directly that for $k\in \ZZ$, the $k^{\text{th}}$ cohomology sheaf of the object $\Phi(L_i^{-1}\vert_{\ell_i})$ is 
\[
H^k\big(\Phi\big(L_i^{-1}\vert_{\ell_i}\big)\big) = H^k\big(\mathbf{R}\Hom(T^\vee,L_i^{-1}\vert_{\ell_i})\big) = H^k\Big(\mathbf{R}\Gamma\big(T\otimes L_i^{-1}\vert_{\ell_i}\big)\Big) \cong \bigoplus_{j\in \qpol_0} H^k(L_j\otimes L_i^{-1}\vert_{\ell_i}).
\]
For any vertex $j\neq i$, Lemma~\ref{lem:floppingcurve} gives $L_j\otimes L_i^{-1}\vert_{\ell_i}\cong \mathscr{O}_{\ell_i}(-1)$ which has vanishing cohomology. Thus, the only term in the direct sum that survives is the $i=j$ term, leaving $H^k(\Phi\big(L_i^{-1}\vert_{\ell_i}\big) \cong H^k(\mathscr{O}_{\ell_i})\otimes_\kk e_i$, where we tensor with $e_i$ to keep track of the $A$-module structure. Therefore
\[
H^k(\Phi\big(L_i^{-1}\vert_{\ell_i}\big) \cong \left\{\begin{array}{cr} \kk e_i & \text{for }k=0 \\ 0 & \text{for }k\neq 0.\end{array}\right.
\]
Thus, $\Phi(L_i^{-1}\vert_{\ell_i})$ is quasi-isomorphic to $S_i=\kk e_i$ as a complex concentrated in degree zero, giving  $\Psi(S_i) \cong L_i^{-1}\vert_{\ell_i}$ as required.
\end{proof}

\begin{proof}[Proof of Theorem~\ref{thm:intromain1}]
In light of Propositions~\ref{prop:dualPsiS0}, \ref{prop:PsiS0}, it remains to prove part \one. Let $i\in \qpol_0$ be nonzero. Proposition~\ref{prop:cohomology} implies that $H^{k}(\Psi(S_i)) = 0$ for $k\not\in \{-1,0\}$. If $H^{0}(\Psi(S_i))\neq 0$, Proposition~\ref{prop:type0or1wall} shows that $\Psi(S_i)$ is quasi-isomorphic to the sheaf $L_i^{-1}\vert_{Z_i}$ where $Z_i$ is the unstable locus for the wall $\overline{C}\cap S_i^\perp$. In particular, $H^{-1}(\Psi(S_i)) = 0$ as required.
\end{proof}

\begin{example}
\label{exa:longhex}
The dimer model $\qpol$ shown in Figure~\ref{fig:longHex} encodes the Jacobian algebra $A$ whose centre defines the Gorenstein semigroup algebra $\kk [\sigma^{\vee}\cap \ZZ^3]$, where $\sigma$ is the cone generated by vectors $v_1=(1,-1,1)$, $v_2=(1,0,1)$, $v_3=(-1,2,1)$, $v_4=(-2,2,1)$, $v_5=(-2,1,1)$, $v_6=(0,-1,1)$.
\begin{figure}[ht!]
\centering
\begin{pspicture}(1.2,-0.5)(6,8)
\psset{unit=0.6cm}
\cnodeput(0,0){A1}{\psscalebox{0.6}{$\mathbf{0}$}}
\cnodeput(12.4,0){A2}{\psscalebox{0.6}{$\mathbf{0}$}}
\cnodeput(0,12.4){A3}{\psscalebox{0.6}{$\mathbf{0}$}}
\cnodeput(12.4,12.4){A4}{\psscalebox{0.6}{$\mathbf{0}$}}
\cnodeput(9.95,8.65){B}{\psscalebox{0.6}{$\mathbf{1}$}}
\cnodeput(6.8,5.6){C}{\psscalebox{0.6}{$\mathbf{2}$}}
\cnodeput(8.2,4.2){D}{\psscalebox{0.6}{$\mathbf{3}$}}
\cnodeput(1.8,10.6){E}{\psscalebox{0.6}{$\mathbf{4}$}}
\cnodeput(9.6,2.8){F}{\psscalebox{0.6}{$\mathbf{5}$}}
\cnodeput(1.25,4.95){G}{\psscalebox{0.6}{$\mathbf{6}$}}
\cnodeput(3.75,2.45){H}{\psscalebox{0.6}{$\mathbf{7}$}}
\cnodeput(7.45,11.15){I}{\psscalebox{0.6}{$\mathbf{8}$}}
\cnodeput(4.2,8.2){J}{\psscalebox{0.6}{$\mathbf{9}$}}

\rput(-1,-1){\pnode{AA1}{}}
\rput(13.4,-1){\pnode{AA2}{}}
\rput(-1,13.4){\pnode{AA3}{}}
\rput(13.4,13.4){\pnode{AA4}{}}

\rput(2.83,-1){\pnode{BB}{}}
\rput(3.65,-1){\pnode{CC}{}}
\rput(4,-1){\pnode{DD}{}}
\rput(4.54,-1){\pnode{EE}{}}
\rput(7.67,-1){\pnode{FF}{}}

\rput(13.4,4.73){\pnode{GG}{}}
\rput(13.4,7.86){\pnode{HH}{}}
\rput(13.4,8.4){\pnode{II}{}}
\rput(13.4,8.75){\pnode{JJ}{}}
\rput(13.4,9.57){\pnode{KK}{}}

\rput(3.37,13.4){\pnode{PP}{}}
\rput(4.95,13.4){\pnode{OO}{}}
\rput(6,13.4){\pnode{NN}{}}
\rput(7.6,13.4){\pnode{MM}{}}
\rput(8.95,13.4){\pnode{LL}{}}

\rput(-1,3.45){\pnode{UU}{}}
\rput(-1,4.8){\pnode{TT}{}}
\rput(-1,6.4){\pnode{SS}{}}
\rput(-1,7.45){\pnode{RR}{}}
\rput(-1,9.03){\pnode{QQ}{}}

\rput(1,-1){\pnode{XX1}{}}
\rput(-1,1){\pnode{WW1}{}}
\rput(-1,-0.25){\pnode{YY1}{}}
\rput(-0.65,-1){\pnode{ZZ1}{}}

\rput(11.75,-1){\pnode{WW2}{}}
\rput(8.44,-1){\pnode{XX2}{}}
\rput(13.4,0.65){\pnode{YY2}{}}
\rput(13.4,3.96){\pnode{ZZ2}{}}

\rput(0.25,13.4){\pnode{WW3}{}}
\rput(1.53,13.4){\pnode{XX3}{}}
\rput(-1,12.15){\pnode{YY3}{}}
\rput(-1,10.87){\pnode{ZZ3}{}}

\rput(13.4,11.4){\pnode{WW4}{}}
\rput(11.4,13.4){\pnode{XX4}{}}
\rput(13.4,13.05){\pnode{YY4}{}}
\rput(12.65,13.4){\pnode{ZZ4}{}}

\ncline[linewidth=0.5pt]{-}{AA1}{AA2}
\ncline[linewidth=0.5pt]{-}{AA2}{AA4}
\ncline[linewidth=0.5pt]{-}{AA4}{AA3}
\ncline[linewidth=0.5pt]{-}{AA3}{AA1}

\ncline[linewidth=0.4pt,linestyle=dashed]{-}{A1}{A2}
\ncline[linewidth=0.4pt,linestyle=dashed]{-}{A2}{A4}
\ncline[linewidth=0.4pt,linestyle=dashed]{-}{A4}{A3}
\ncline[linewidth=0.4pt,linestyle=dashed]{-}{A3}{A1}

\ncline[linewidth=1pt]{<-}{A1}{WW1}
\ncline[linewidth=1pt]{-}{A1}{XX1}
\ncline[linewidth=1pt]{-}{A1}{YY1}
\ncline[linewidth=1pt]{<-}{A1}{ZZ1}

\ncline[linewidth=1pt]{-}{A2}{AA2}
\ncline[linewidth=1pt]{<-}{A2}{YY2}
\ncline[linewidth=1pt]{-}{A2}{ZZ2}
\ncline[linewidth=1pt]{<-}{A2}{WW2}
\ncline[linewidth=1pt]{-}{A2}{XX2}

\ncline[linewidth=1pt]{->}{AA3}{A3}
\ncline[linewidth=1pt]{-}{A3}{WW3}
\ncline[linewidth=1pt]{<-}{A3}{XX3}
\ncline[linewidth=1pt]{-}{A3}{YY3}
\ncline[linewidth=1pt]{<-}{A3}{ZZ3}

\ncline[linewidth=1pt]{-}{A4}{WW4}
\ncline[linewidth=1pt]{<-}{A4}{XX4}
\ncline[linewidth=1pt]{<-}{A4}{YY4}
\ncline[linewidth=1pt]{-}{A4}{ZZ4}

\ncline[linewidth=1pt]{->}{A1}{G}\lput*{:0}{\psscalebox{0.6}{$\mathbf{34}$}}
\ncline[linewidth=1pt]{->}{A3}{E}\lput*{:0}{\psscalebox{0.6}{$\mathbf{126}$}}
\ncline[linewidth=1pt]{->}{A4}{I}\lput*{:180}{\psscalebox{0.6}{$\mathbf{45}$}}
\ncline[linewidth=1pt]{->}{H}{A1}\lput*{:180}{\psscalebox{0.6}{$\mathbf{5789}$}}
\ncline[linewidth=1pt]{->}{F}{A2}\lput*{:0}{\psscalebox{0.6}{$\mathbf{12678910}$}}
\ncline[linewidth=1pt]{->}{B}{A4}\lput*{:0}{\psscalebox{0.6}{$\mathbf{38910}$}}

\ncline[linewidth=1pt]{->}{B}{C}\lput*{:180}{\psscalebox{0.6}{$\mathbf{6}$}}
\ncline[linewidth=1pt]{->}{J}{B}\lput*{:0}{\psscalebox{0.6}{$\mathbf{12}$}}
\ncline[linewidth=1pt]{->}{KK}{B}\lput*{:180}{\psscalebox{0.6}{$\mathbf{457}$}}

\ncline[linewidth=1pt]{->}{C}{J}\lput*{:180}{\psscalebox{0.6}{$\mathbf{34578910}$}}
\ncline[linewidth=1pt]{->}{H}{C}\lput*{:0}{\psscalebox{0.6}{$\mathbf{2}$}}
\ncline[linewidth=1pt]{->}{D}{C}\lput*{:180}(0.6){\psscalebox{0.5}{$\mathbf{4}$}}
\ncline[linewidth=1pt]{-}{C}{JJ}\lput*{:0}{\psscalebox{0.6}{$\mathbf{1238910}$}}
\ncline[linewidth=1pt]{-}{C}{CC}\lput*{:180}{\psscalebox{0.6}{$\mathbf{156789}$}}

\ncline[linewidth=1pt]{<-}{F}{D}\lput*{:180}(0.6){\psscalebox{0.6}{$\mathbf{19}$}}
\ncline[linewidth=1pt]{->}{DD}{D}\lput*{:0}{\psscalebox{0.6}{$\mathbf{2310}$}}
\ncline[linewidth=1pt]{->}{II}{D}\lput*{:180}{\psscalebox{0.6}{$\mathbf{567}$}}

\ncline[linewidth=1pt]{->}{E}{G}\lput*{:180}{\psscalebox{0.6}{$\mathbf{1679}$}}
\ncline[linewidth=1pt]{->}{E}{I}\lput*{:0}{\psscalebox{0.6}{$\mathbf{12910}$}}
\ncline[linewidth=1pt]{-}{E}{NN}
\ncline[linewidth=1pt]{-}{E}{PP}
\ncline[linewidth=1pt]{-}{E}{SS}
\ncline[linewidth=1pt]{-}{E}{QQ}
\ncline[linewidth=1pt]{->}{J}{E}\lput*{:180}{\psscalebox{0.6}{$\mathbf{3458}$}}
\ncline[linewidth=1pt]{->}{OO}{E}
\ncline[linewidth=1pt]{->}{MM}{E}
\ncline[linewidth=1pt]{->}{RR}{E}
\ncline[linewidth=1pt]{->}{TT}{E}

\ncline[linewidth=1pt]{-}{F}{EE}\lput*{:180}{\psscalebox{0.6}{$\mathbf{45678}$}}
\ncline[linewidth=1pt]{-}{F}{HH}\lput*{:0}{\psscalebox{0.6}{$\mathbf{567}$}}
\ncline[linewidth=1pt]{<-}{F}{FF}\lput*{:180}{\psscalebox{0.6}{$\mathbf{3}$}}
\ncline[linewidth=1pt]{<-}{F}{GG}\lput*{:0}{\psscalebox{0.6}{$\mathbf{5}$}}

\ncline[linewidth=1pt]{->}{G}{J}\lput*{:0}{\psscalebox{0.6}{$\mathbf{210}$}}
\ncline[linewidth=1pt]{-}{G}{UU}

\ncline[linewidth=1pt]{->}{BB}{H}\lput*{:0}{\psscalebox{0.6}{$\mathbf{3410}$}}

\ncline[linewidth=1pt]{->}{J}{H}\lput*{:180}{\psscalebox{0.6}{$\mathbf{16}$}}

\ncline[linewidth=1pt]{->}{I}{J}\lput*{:180}{\psscalebox{0.6}{$\mathbf{67}$}}
\ncline[linewidth=1pt]{-}{I}{LL}
\end{pspicture}
\caption{A fundamental domain for a dimer model $\qpol$}
\label{fig:longHex}
\end{figure}
 The height one slice of the cone $\sigma$ is the lattice polygon shown in Figure~\ref{fig:toricfanX}, and the height one slice of the fan $\Sigma$ defining the toric variety $Y=\mathcal{M}_\vartheta$ for the stability parameter $\vartheta=(-9,1,1,1,1,1,1,1,1,1)$ is shown in Figure~\ref{fig:toricfanY}. For $1 \leq \rho \leq 10$, let $E_{\rho}$ denote the divisor in $Y$ corresponding to the ray of $\Sigma$ generated by $v_{\rho}$. For each $a\in \qpol_1$, the map $\phi(a)\colon L_{\tail(a)}\to L_{\head(a)}$ is multiplication by a section, and we label each arrow in Figure~\ref{fig:longHex} with the divisor of zeroes of the section; we use the shorthand $126=E_1+E_2+E_6$. 
\begin{figure}[!ht]
   \centering
      \subfigure[]{
      \label{fig:toricfanX}
       \psset{unit=1cm}
     \begin{pspicture}(0,-0.5)(3,3.2)
       \psline{*-*}(3,0)(3,1)
        \psline{-*}(3,1)(2,2)
         \psline{-*}(2,2)(1,3)
          \psline{-*}(1,3)(0,3)
           \psline{-*}(0,3)(0,2)
            \psline{-*}(0,2)(1,1)
             \psline{-*}(1,1)(2,0)
              \psline{-}(2,0)(3,0)
       \psdot(0,0)
       \psdot(1,0)
       \psdot(2,0)
       \psdot(3,0)
       \psdot(0,1)
       \psdot(1,1)
       \psdot(2,1)
       \psdot(3,1)
        \psdot(0,2)
       \psdot(1,2)
       \psdot(2,2)
       \psdot(3,2)
       \psdot(0,3)
       \psdot(1,3)
       \psdot(2,3)
       \psdot(3,3)
      \rput(3.2,-0.2){$v_1$}
       \rput(3.2,1.2){$v_2$}
           \rput(1.2,3.2){$v_3$}
               \rput(-0.2,3.2){$v_4$}
                   \rput(-0.2,1.8){$v_5$}
                       \rput(1.8,-0.2){$v_6$}
       \end{pspicture}
       }
      \qquad  \qquad
      \subfigure[]{
       \label{fig:toricfanY}
              \psset{unit=1cm}
     \begin{pspicture}(0,-0.5)(3,3)
       \psline{*-*}(3,0)(3,1)
        \psline{-*}(3,1)(2,2)
         \psline{-*}(2,2)(1,3)
          \psline{-*}(1,3)(0,3)
           \psline{-*}(0,3)(0,2)
            \psline{-*}(0,2)(1,1)
             \psline{-*}(1,1)(2,0)
              \psline{-}(2,0)(3,0)
               \psline{*-}(2,1)(3,0)
                \psline{*-}(2,1)(3,1)
                 \psline{*-}(2,1)(2,2)
                  \psline{*-}(2,1)(1,2)
                    \psline{*-}(2,1)(1,1)
                      \psline{*-}(2,1)(2,0)
                       \psline{*-}(1,2)(2,2)
                        \psline{*-}(1,2)(1,3)
                         \psline{*-}(1,2)(0,3)
                          \psline{*-}(1,2)(0,2)
                           \psline{*-}(1,2)(1,1)
       \psdot(0,0)
       \psdot(1,0)
       \psdot(2,0)
       \psdot(3,0)
       \psdot(0,1)
       \psdot(1,1)
       \psdot(2,1)
       \psdot(3,1)
        \psdot(0,2)
       \psdot(1,2)
       \psdot(2,2)
       \psdot(3,2)
       \psdot(0,3)
       \psdot(1,3)
       \psdot(2,3)
       \psdot(3,3)
      \rput(3.2,-0.2){$v_1$}
       \rput(3.2,1.2){$v_2$}
           \rput(1.2,3.2){$v_3$}
               \rput(-0.2,3.2){$v_4$}
                   \rput(-0.2,1.8){$v_5$}
                       \rput(1.8,-0.2){$v_6$}
                        \rput(0.8,0.8){$v_7$}
                          \rput(1.25,2.2){$v_8$}
                         \rput(2.25,1.2){$v_9$}
                          \rput(2.2,2.2){$v_{10}$}
                          \end{pspicture}        }
           \caption{(a) Lattice polygon defining $X$; (b) Triangulation defining $Y$.}
           \label{fig:Sigma}
  \end{figure}
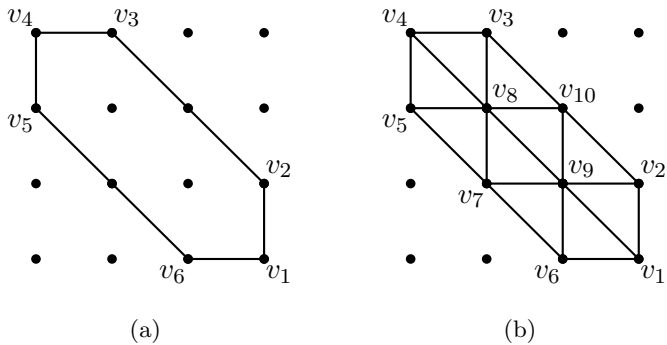
  
These labelling divisors enable us to compute very simply all ten torus-invariant $\vartheta$-stable $A$-modules, one for each basic triangle in Figure~\ref{fig:toricfanY}. For example, for the triangle $\tau$ with vertices $\{v_1, v_2, v_{9}\}$, the torus-invariant point is $E_1\cap E_2\cap E_9$, so to obtain the corresponding torus-invariant $\vartheta$-stable $A$-module we need only set to zero every arrow in $\qpol$ whose labelling divisor contains $E_1$, $E_2$ or $E_9$. Vertices 2 and 7 are sinks for the resulting quiver, so $S_2$ and $S_7$ are submodules of the corresponding torus-invariant $\vartheta$-stable $A$-module. In fact, $S_2$ is a submodule of every torus-invariant $\vartheta$-stable $A$-module, so Propositions~\ref{prop:TFAE} and \ref{prop:type0or1wall} give 
\[
\Psi(S_2) \cong L_2^{-1}\vert_{E_8\cup E_9}.
\]
 In particular, the wall of type zero of the form $\overline{C}\cap S_2^\perp$ has unstable locus equal to the reducible divisor $E_8\cup E_9$. On the other hand, the triangles with vertices $\{v_1, v_2, v_{9}\}$ and $\{v_2, v_9, v_{10}\}$ are the only ones for which the corresponding $\vartheta$-stable $A$-module has $S_7$ as a submodule, so Propositions~\ref{prop:TFAE} and \ref{prop:type0or1wall} imply that 
\[
\Psi(S_7) \cong L_7^{-1}\vert_{E_2\cap E_9}.
\]
Similarly, we have $\Psi(S_1) \cong L_1^{-1}\vert_{E_6\cap E_9}$ and $\Psi(S_5)\cong L_5^{-1}\vert_{E_8}$.
\end{example}

\section{Nonvanishing cohomology in degree minus one}
 In this section we investigate the sheaf $H^{-1}(\Psi(S_i))$ for $i\in \qpol_0$. We recall the explicit calculation of a filtration of this sheaf, and we use this to show that the every irreducible component in the support of $H^{-1}(\Psi(S_i))$ is a compact, connected torus-invariant divisor in $Y$. 

\subsection{The cohomology of wheels}
Once and for all, fix $i\in \qpol_0$ and let $m:=m(i)$ denote the number of arrows in $\qpol$ with head (equivalently, tail) at $i$.  List the arrows in $\qpol$ with tail at $i\in \qpol_0$ in cyclic order as $a_1,\dots, a_m$, and list the arrows with head at $i\in \qpol_0$ in cyclic order as $b_1,\dots, b_m$ so that for all $1\leq j\leq m$, the arrow $b_j$ lies between $a_j$ and $a_{j+1}$ in the universal cover as shown in \cite[Figure (7.5)]{Broomhead12}; we add indices modulo $m$. With this notation, the complex from Lemma~\ref{lem:imageSi} can be written as
 \begin{equation}
\label{eqn:diagram}
\begin{split}
    \centering    
         \psset{unit=0.45cm}
     \begin{pspicture}(0,-1.2)(25,10.5)
\cnodeput*(0,4.4){A}{$L_i^{-1}$} 
\cnodeput*(8,8.8){B}{$L_{\tail(b_1)}^{-1}$}
\cnodeput*(8,6.6){C}{$L_{\tail(b_2)}^{-1}$} 
\cnodeput*(8,4.4){D}{$L_{\tail(b_3)}^{-1}$}
\cnodeput*(8,2.2){S}{$\vdots$}
\cnodeput*(8,0){E}{$L_{\tail(b_m)}^{-1}$}
\cnodeput*(18,8.8){F}{$L_{\head(a_1)}^{-1}$}
\cnodeput*(18,6.6){G}{$L_{\head(a_2)}^{-1}$}
\cnodeput*(18,4.4){H}{$L_{\head(a_3)}^{-1}$}
\cnodeput*(18,2.2){T}{$\vdots$}
\cnodeput*(18,0){I}{$L_{\head(a_m)}^{-1}$}
\cnodeput*(26,4.4){J}{$L_i^{-1}$}
\psset{nodesep=1pt}
   \ncline{->}{A}{B}\lput*{:U}(0.6){$\scriptstyle{D_{1,2}}$}
   \ncline{->}{A}{C}\lput*{:U}(0.6){$\scriptstyle{D_{2,3}}$}
   \ncline{->}{A}{D}\lput*{:U}(0.6){$\scriptstyle{D_{3,4}}$}
   \ncline{->}{A}{E}\lput*{:U}(0.6){$\scriptstyle{D_{m,1}}$}
 \ncline{->}{B}{F}\lput*{:U}(0.3){$\scriptstyle{D^2_1}$}
  \ncline{->}{B}{G}\lput*{:U}(0.75){$\scriptstyle{D^1_2}$}
  \ncline{->}{C}{G}\lput*{:U}(0.3){$\scriptstyle{D^3_2}$}
  \ncline{->}{C}{H}\lput*{:U}(0.75){$\scriptstyle{D^2_3}$}
\ncline{->}{D}{H}\lput*{:U}(0.3){$\scriptstyle{D^4_3}$}
\ncline{->}{E}{I}\lput*{:U}(0.5){$\scriptstyle{D^1_m}$}
\nccurve[angleA=-40,angleB=140]{->}{E}{F}\lput*{:U}(0.4){$\scriptscriptstyle{D^m_1}$}
 \ncline{->}{F}{J}\lput*{:U}(0.4){$\scriptstyle{D^{1}}$}
   \ncline{->}{G}{J}\lput*{:U}(0.4){$\scriptstyle{D^{2}}$}
   \ncline{->}{H}{J}\lput*{:U}(0.4){$\scriptstyle{D^{3}}$}
   \ncline{->}{I}{J}\lput*{:U}(0.4){$\scriptstyle{D^{m}}$}
       \end{pspicture}
 \end{split}
 \end{equation}
 where we \emph{label} each map of line bundles by an effective, torus-invariant divisor $D^j, D_j^{j+1}, D_{j+1}^j$ or $D_{j,j+1}$ for $1\leq j\leq m$ to indicate that the morphism is multiplication by a section whose divisor of zeroes is the given divisor. To preserve the cyclic order, note that $D^j$ is the label for $\phi^\vee(a_j^{\textrm{op}})\colon L_{\head(a_j)}^{-1}\to L_i^{-1}$ and $D_{j,j+1}$ is the label for $\phi^\vee(b_j^{\textrm{op}})\colon L_{i}^{-1}\to L_{\tail(b_j)}^{-1}$. The relations give 
\begin{align}
D^j_{j+1}+D^{j+1} & = D_j^{j+1}+D^j  \label{eqn:relations1}\\
D^{j-1}_j+D_{j-1,j} & = D_j^{j+1}+D_{j,j+1} \label{eqn:relations2}
\end{align}
for $1\leq j\leq m$. 

List the transpositions of $m$ letters as $\tau_1= (\mu_1,\nu_1),\dots, \tau_n=(\mu_n,\nu_n)$ for $n= \binom{m}{2}$ as follows. First list the transpositions of adjacent letters $\tau_j=(j,j+1)$ for $1 \leq j \leq m-1$. Set $\tau_m=(1,m)$, then list all remaining transpositions that involve $1$ as $\tau_j=(1,j-m+2)$ for $m+1 \leq j \leq 2m-3$. Finally list all remaining transpositions lexicographically, so $\tau_i=(\mu_i,\nu_i)$ precedes $\tau_j=(\mu_j,\nu_j)$ if and only if $\mu_i < \mu_j$ or $\mu_i=\mu_j$ and $\nu_i < \nu_j$. Craw--Quintero-V\'{e}lez~\cite{CrawQuinteroVelez12b} introduced this order to describe the cohomology sheaf in degree $-1$ of a complex of the form \eqref{eqn:diagram}. 
 
\begin{proposition}
\label{prop:H-1}
For $i\in \qpol_0$, the support of $H^{-1}(\Psi(S_i))$ is of the form $\bigcup_{1\leq j\leq n} Z_j(i)$, where each subscheme $Z_j(i) \subset Y$ is obtained as the scheme-theoretic intersection of a set of effective divisors in $Y$ determined by the value of $j$ as follows:
\begin{enumerate}
\item[\one] for $1\leq j\leq m$, the set comprises only two divisors, namely $\gcd(D_{j+1}^j,D^{j+1}_j)$ and 
\[
\lcm\big(D^1,\dots,D^m,\gcd(D_{j+2}^{j+1},D^{j+2}_{j+1}),\dots,\gcd(D_{1}^{m},D^{1}_{m})\big)-\lcm(D^j,D^{j+1});
\]
\item[\two] for $m+1 \leq j \leq 2m-3$, the divisors are $\lcm(D^1,D^{\nu_j},D^{\nu_j+1},\dots,D^m)-\lcm(D^1, D^{\nu_j})$ and $\lcm(D^1,D^{\nu_j-1},D^{\nu_j})-\lcm(D^1,D^{\nu_j})$;
\item[\three] for $2m-2\leq j \leq n$, there are $\mu_j$ such divisors, namely $\lcm(D^{\mu},D^{\mu_j},D^{\nu_j})-\lcm(D^{\mu_j},D^{\nu_j})$ for $\mu \in \{1, \dots, \mu_j-1\}\cup\{\nu_j-1\}$.
\end{enumerate}
\end{proposition}
\begin{proof}
 Apply \cite[Theorem~1.1(2)]{CrawQuinteroVelez12b} to the complex $\Psi(S_i)$ from Lemma~\ref{lem:imageSi} labelled with divisors as shown in \eqref{eqn:diagram}. The filtration $\im(d^2)=F^0 \subseteq F^1\subseteq \cdots \subseteq F^n=\ker(d^1)$ on $H^{-1}(\Psi(S_i))$ is such that for $1\leq j \leq n$, the quotient $F^j/F^{j-1}$ has support equal to the subscheme $Z_j(i)$.
\end{proof}

 This enables us generalise a result of Logvinenko~\cite[Lemma 3.4]{Logvinenko10} to the dimer setting:

\begin{corollary}\label{cor:dim1or2}
 For any $i\in \qpol_0$, the support of $\Psi(S_i)$ is connected, and each of its irreducible components is a compact torus-invariant curve or surface. 
\end{corollary}
\begin{proof}
The functor $\Psi$ is an equivalence, so $\End_{D^b(\coh(Y))}(\Psi(S_i)) \cong \End_{D^b(\modA)}(S_i)\cong \mathbb{C}$. This implies that the support of $\Psi(S_i)$ is compact and connected. Propositions~\ref{prop:cohomology} and \ref{prop:H-1} show that for any $k\in \ZZ$, the support of $H^k(\Psi(S_i))$ is obtained as an intersection of certain torus-invariant divisors in $Y$. Every irreducible component is therefore a torus-invariant subvariety of dimension at most two. In fact, a torus-invariant point never occurs, because if the support of $\Psi(S_i)$ were a torus-invariant point $y\in Y$, then $S_i = \Phi(\Psi(S_i)) = \Phi(\mathscr{O}_y) = V_y$ is the $\vartheta$-stable $A$-module corresponding to $y\in Y$, a contradiction.
\end{proof}

 To prove Theorem~\ref{thm:intromain2}, it remains to show that each component of the support of $H^{-1}(\Psi(S_i))$ has dimension two. It is convenient to depict the diagram of line bundles \eqref{eqn:diagram} as a planar picture as in Figure~\ref{fig:wheel}; this is the \emph{wheel} of line bundles, denoted $W_i$ for $i\in \qpol_0$. 
\begin{figure}[!ht]
    \centering    
         \psset{unit=0.5cm}
     \begin{pspicture}(-15,-8)(15,8.2)
\cnodeput*(0,0){A}{\psscalebox{0.8}{$L_i^{-1}$}}
\rput(8,0){\ovalnode*{B}{\psscalebox{0.8}{$L_{\head(a_1)}^{-1}$}}}
\rput(7,4){\ovalnode*{C}{\psscalebox{0.8}{$L_{\tail(b_1)}^{-1}$}}}
\rput(4,7){\ovalnode*{D}{\psscalebox{0.8}{$L_{\head(a_2)}^{-1}$}}}
\rput(0,8){\ovalnode*{E}{\psscalebox{0.8}{$L_{\tail(b_2)}^{-1}$}}}
\rput*(-4,7){\ovalnode*{F}{\psscalebox{0.8}{$L_{\head(a_3)}^{-1}$}}}
\rput*(-7,4){\ovalnode*{G}{\psscalebox{0.8}{$L_{\tail(b_3)}^{-1}$}}}
\rput*(-8,0){\ovalnode*{H}{\psscalebox{0.8}{$L_{\head(a_4)}^{-1}$}}}
\cnodeput*(-7,-4){I}{\;}
\cnodeput*(-4,-7){J}{\;\;\;\;\;}
\cnodeput*(0,-8){K}{\;}
\rput(4,-7){\ovalnode*{L}{\psscalebox{0.8}{$L_{\head(a_m)}^{-1}$}}}
\rput(7,-4){\ovalnode*{M}{\psscalebox{0.8}{$L_{\tail(b_m)}^{-1}$}}}
\psset{nodesep=1pt}
\ncline[linewidth=.6mm,linestyle=dotted,dotsep=15pt]{-}{I}{J}
\ncline[linewidth=.6mm,linestyle=dotted,dotsep=15pt]{-}{K}{J}  
   \ncline{<-}{A}{B}\lput*{:U}(0.6){$\scriptstyle{D^1}$}
   \ncline{<-}{A}{D}\lput*{:-60}(0.6){$\scriptstyle{D^2}$}
     \ncline{<-}{A}{F}\lput*{:-120}(0.6){$\scriptstyle{D^3}$}
       \ncline{<-}{A}{H}\lput*{:-180}(0.6){$\scriptstyle{D^4}$}
         \ncline{<-}{A}{L}\lput*{:60}(0.6){$\scriptstyle{D^m}$}
    \ncline{->}{A}{C}\lput*{:-30}(0.6){$\scriptstyle{D_{1,2}}$}
      \ncline{->}{A}{E}\lput*{:-90}(0.6){$\scriptstyle{D_{2,3}}$}
        \ncline{->}{A}{G}\lput*{:-150}(0.6){$\scriptstyle{D_{3,4}}$}
            \ncline{->}{A}{M}\lput*{:30}(0.6){$\scriptstyle{D_{m,1}}$}
    \ncline{<-}{B}{C}\bput*{:-105}{$\scriptstyle{D^2_1}$} 
     \ncline{->}{C}{D}\bput*{:-135}{$\scriptstyle{D^1_2}$}         
    \ncline{<-}{D}{E}\bput*{:-165}{$\scriptstyle{D^3_2}$}         
     \ncline{->}{E}{F}\bput*{:-195}{$\scriptstyle{D^2_3}$}           
 \ncline{<-}{F}{G}\bput*{:-225}{$\scriptstyle{D^4_3}$}  
     \ncline{->}{G}{H}\bput*{:-255}{$\scriptstyle{D^3_4}$}   
      \ncline{<-}{H}{I}  
 \ncline{->}{K}{L}    
     \ncline{<-}{L}{M}\bput*{:-45}{$\scriptstyle{D^1_m}$}   
      \ncline{->}{M}{B}  \bput*{:-75}{$\scriptstyle{D^m_1}$} 
        \end{pspicture}
\caption{Diagram \eqref{eqn:diagram} shown as a wheel of line bundles}
     \label{fig:wheel}
    \end{figure}
The \emph{in-spokes} in $W_i$ are the arrows pointing towards $L_i^{-1}$, each of which is labelled by an \emph{in-spoke divisor} $D^j$ for some $1\leq j\leq m$. The \emph{out-spokes} are the arrows pointing away from $L_{i}^{-1}$, each labelled by an \emph{out-spoke divisor} $D_{j,j+1}$ for $1\leq j\leq m$. Also, for $1\leq j\leq m$, the \emph{rim divisors} $D_{j+1}^j$ and $D_j^{j+1}$ label the arrows from $L_{\tail(b_j)}^{-1}$ to $L_{\head(a_j)}^{-1}$ and $L_{\head(a_{j+1})}^{-1}$ respectively that form the rim of the wheel.

For any divisor $D$ that labels an arrow in the wheel, we say that a compact, irreducible torus-invariant divisor $E\subset Y$ is \emph{contained in} $D$, denoted $E\subseteq D$, if $\supp(E)\subseteq \supp(D)$. 

\begin{lemma}
\label{lem:H-1equivalent}
 For $i\in \qpol_0$, a compact, irreducible torus-invariant divisor $E$ is contained in the support of $H^{-1}(\Psi(S_i))$ if and only if:
\begin{enumerate}
\item[\one] (for $1\leq j\leq m$) $E \subseteq D^j_{j+1}$, $E\subseteq D^{j+1}_j$, $E \not\subseteq D^j$, $E \not\subseteq D^{j+1}$ and either $E \subseteq D^{\mu}$ for some $\mu \in \{1,\dots, m\}\setminus \{j,j+1\}$ or $E \subseteq \gcd(D^{\mu}_{\mu+1},D^{\mu+1}_{\mu})$ for some $j+1 \leq \mu \leq m$; or 
\item[\two] (for $m+1\leq j\leq 2m-3$) $E \not\subseteq D^1$, $E\not\subseteq D^{\nu_j}$, $E \subseteq D^{\nu_j-1}$ and $E \subseteq D^{\mu}$ for some $\nu_{j}+1 \leq \mu \leq m$; or
\item[\three](for $2m-2\leq j\leq n$) $E \not \subseteq D^{\mu_j}$, $E\not\subseteq D^{\nu_j}$ and $E \subseteq D^{\mu}$ for all $\mu \in \{1,\dots, \mu_j-1  \}\cup \{\nu_j-1\}$.
\end{enumerate}
\end{lemma}
\begin{proof}
Remark~\ref{rem:ishiiueda09}(2) implies that the multiplicity of $E$ in any of the divisors labelling an arrow in diagram \eqref{eqn:diagram} is zero or one. We may therefore rewrite any condition from Proposition~\ref{prop:H-1} requiring $E$ to lie in the difference of two $\lcm$ divisors by the condition that $E$ lies in precisely one of them; for example
\[
E\subseteq \lcm(D^{\mu},D^{\mu_j},D^{\nu_j})-\lcm(D^{\mu_j},D^{\nu_j})\iff E \subseteq  \lcm(D^{\mu},D^{\mu_j},D^{\nu_j})\text{ and }E\not\subseteq \lcm(D^{\mu_j},D^{\nu_j}).
\]
It is then straightforward to deduce Lemma~\ref{lem:H-1equivalent} from Proposition~\ref{prop:H-1}. 
\end{proof}

 Our next goal is to characterise precisely when a given divisor $E$ is not contained in the support of $\Psi(S_i)$ for a nonzero vertex $i$. First we present a technical lemma.

\begin{lemma}
\label{lem:technicalWheels}
Let $i\in \qpol_0$ and let $E$ be a compact, irreducible torus-invariant divisor. Then
\begin{enumerate}
\item[\one] for each cycle in $W_i$ comprising an out-spoke, a rim arrow and an in-spoke, the divisor $E$ is contained in one and only one of the divisors labelling these three arrows.
\item[\two] if $E$ is contained in some in-spoke divisor $D^j$ but not in another $D^k$, then there exist indices $\mu, \nu$ with $k< \mu+1\leq j\leq \nu-1$ such that $E$ is contained in $\gcd(D^{\mu+1}_{\mu}, D^{\nu-1}_{\nu})$ and the in-spoke divisors $D^{\mu+1},\dots,D^{\nu-1}$, but $E$ is not contained in $D^{\mu}$ or $D^{\nu}$. 
\end{enumerate}
\end{lemma}
\begin{proof}
Part \one\ is a restatement of Remark~\ref{rem:ishiiueda09}(2). For part \two, apply \one\ to both cycles containing the in-spoke labelled $D^j$ to obtain $E \subseteq D^{j}_{j-1}$ or $E \subseteq D^{j-1}$, and $E \subseteq D^{j}_{j+1}$ or $E \subseteq D^{j+1}$. We define $\mu$ by considering the first pair. If $E\subseteq D^{j}_{j-1}$ then we set $\mu=j-1$. Otherwise, $E \subset D^{j-1}$. Part \one\ gives either $E \subseteq D^{j-1}_{j-2}$, when we set $\mu=j-2$, or $E \subseteq D^{j-2}$, in which case we repeat the argument as necessary to obtain $\mu \geq k$ such that $E$ is contained in $D^{\mu+1},\dots, D^{j-1}$ but not $D^{\mu}$. We define $\nu$ similarly by considering the second pair, giving $\nu \leq k-1$ such that $D^{j+1},\dots, D^{\nu-1}$ contain $E$, but $D^{\nu}$ does not. In either case, we obtain $E\subseteq \gcd(D^{\mu+1}_{\mu}, D^{\nu-1}_{\nu})$.
\end{proof}

 \begin{proposition}
 \label{prop:vanishing}
 Let $i\in \qpol_0$ be nonzero. A compact, irreducible, torus-invariant divisor $E$ is not contained in $\supp(\Psi(S_i))$ if and only if there exist $1 \leq \mu \leq \nu \leq m$ such that the only divisors containing $E$ that label edges in the wheel $W_i$ are:  the in-spoke divisors $D^{\mu+1}, \dots , D^{\nu-1}$; the out-spoke divisors $D_{\nu,\nu+1},\dots, D_{\mu-1,\mu}$; and the rim divisors $D_\nu^{\nu-1}$ and $D_\mu^{\mu+1}$ (see Figure~\ref{fig:vanishingType}).
 \end{proposition}

\begin{figure}[!ht]
   \centering
       \psset{unit=0.6cm}
     \begin{pspicture}(-6,-7.2)(6,7.5)
\psset{unit=0.75cm}
\cnodeput*(0,0){OO}{\psscalebox{0.6}{$L_i^{-1}$}}
\rput(6,0){\ovalnode*{A}{\psscalebox{0.6}{$L^{-1}_{\tail(b_{\nu-1})}$}}}
\rput(5.71,1.85){\ovalnode*{B}{\psscalebox{0.6}{$L^{-1}_{\head(a_{\nu})}$}}}
\rput(4.85,3.53){\ovalnode*{C}{\psscalebox{0.6}{$L^{-1}_{\tail(b_{\nu-1})}$}}}
\rput(3.53,4.85){\ovalnode*{D}{\psscalebox{0.6}{$L^{-1}_{\head(a_{\nu+1})}$}}}
\rput(1.85,5.71){\ovalnode*{E}{\psscalebox{0.6}{$L^{-1}_{\tail(b_{\nu+1})}$}}}
\ncline[linewidth=1pt,linecolor=lightgray]{->}{OO}{A}
\ncline[linewidth=1pt,linecolor=lightgray]{<-}{OO}{B}
\ncline[linewidth=1pt]{->}{OO}{C}\lput*{:U}{$\scriptstyle{D_{\nu,\nu+1}}$}
\ncline[linewidth=1pt,linecolor=lightgray]{<-}{OO}{D}
\ncline[linewidth=1pt]{->}{OO}{E}\lput*{:U}{$\scriptstyle{D_{\nu+1,\nu+2}}$}
\ncline[linewidth=1pt]{->}{A}{B}\bput*{:270}{$\scriptstyle{D^{\nu-1}_\nu}$}     
\ncline[linewidth=1pt,linecolor=lightgray]{<-}{B}{C}
\ncline[linewidth=1pt,linecolor=lightgray]{->}{C}{D}
\ncline[linewidth=1pt,linecolor=lightgray]{<-}{D}{E}

\rput(-5.71,1.85){\ovalnode*{J}{\psscalebox{0.6}{$L^{-1}_{\tail(b_{\mu-1})}$}}}
\rput(-4.85,3.53){\ovalnode*{I}{\psscalebox{0.6}{$L^{-1}_{\head(a_{\mu-1})}$}}}
\rput(-3.53,4.85){\ovalnode*{H}{\psscalebox{0.6}{$L^{-1}_{\tail(b_{\mu-2})}$}}}
\rput(-1.85,5.71){\ovalnode*{G}{}}
\rput(0,6){\ovalnode*{F}{}}
\ncline[linewidth=1pt]{->}{OO}{H}\lput*{:180}{$\scriptstyle{D_{\mu-2,\mu-1}}$}
\ncline[linewidth=1pt,linecolor=lightgray]{<-}{OO}{I}
\ncline[linewidth=1pt]{->}{OO}{J}\lput*{:180}{$\scriptstyle{D_{\mu-1,\mu}}$}
\ncline[linewidth=1pt,linecolor=lightgray]{->}{E}{F}
\ncline[linewidth=.5mm,linestyle=dotted,dotsep=10pt]{-}{F}{G}
\ncline[linewidth=1pt,linecolor=lightgray]{<-}{G}{H}
\ncline[linewidth=1pt,linecolor=lightgray]{->}{H}{I}
\ncline[linewidth=1pt,linecolor=lightgray]{<-}{I}{J}

\rput(-6,0){\ovalnode*{K}{\psscalebox{0.6}{$L^{-1}_{\head(a_{\mu})}$}}}
\rput(-5.71,-1.85){\ovalnode*{L}{\psscalebox{0.6}{$L^{-1}_{\tail(b_{\mu})}$}}}
\rput(-4.85,-3.53){\ovalnode*{M}{\psscalebox{0.6}{$L^{-1}_{\head(a_{\mu+1})}$}}}
\rput(-3.53,-4.85){\ovalnode*{N}{\psscalebox{0.6}{$L^{-1}_{\tail(b_{\mu+1})}$}}}
\rput(-1.85,-5.71){\ovalnode*{O}{\psscalebox{0.6}{$L^{-1}_{\head(a_{\mu+2})}$}}}
\ncline[linewidth=1pt,linecolor=lightgray]{<-}{OO}{K}
\ncline[linewidth=1pt,linecolor=lightgray]{->}{OO}{L}
\ncline[linewidth=1pt]{<-}{OO}{M}\lput*{:180}{$\scriptstyle{D^{\mu+1}}$}
\ncline[linewidth=1pt,linecolor=lightgray]{->}{OO}{N}
\ncline[linewidth=1pt]{<-}{OO}{O}\lput*{:180}{$\scriptstyle{D^{\mu+2}}$}
\ncline[linewidth=1pt]{<-}{K}{L}\bput*{:90}{$\scriptstyle{D^{\mu+1}_\mu}$}     
\ncline[linewidth=1pt,linecolor=lightgray]{->}{L}{M}
\ncline[linewidth=1pt,linecolor=lightgray]{<-}{M}{N}
\ncline[linewidth=1pt,linecolor=lightgray]{->}{N}{O}
\ncline[linewidth=1pt,linecolor=lightgray]{<-}{K}{J}

\rput(5.71,-1.85){\ovalnode*{T}{\psscalebox{0.6}{$L^{-1}_{\head(a_{\nu-1})}$}}}
\rput(4.85,-3.53){\ovalnode*{S}{\psscalebox{0.6}{$L^{-1}_{\tail(b_{\nu -2})}$}}}
\rput(3.53,-4.85){\ovalnode*{R}{\psscalebox{0.6}{$L^{-1}_{\head(a_{\nu-2})}$}}}
\rput(1.85,-5.71){\ovalnode*{Q}{}}
\rput(0,-6){\ovalnode*{P}{}}
\ncline[linewidth=1pt]{<-}{OO}{R}\lput*{:U}{$\scriptstyle{D^{\nu-2}}$}
\ncline[linewidth=1pt,linecolor=lightgray]{->}{OO}{S}
\ncline[linewidth=1pt]{<-}{OO}{T}\lput*{:U}{$\scriptstyle{D^{\nu-1}}$}
\ncline[linewidth=1pt,linecolor=lightgray]{<-}{O}{P}
\ncline[linewidth=.5mm,linestyle=dotted,dotsep=10pt]{-}{P}{Q}
\ncline[linewidth=1pt,linecolor=lightgray]{->}{Q}{R}
\ncline[linewidth=1pt,linecolor=lightgray]{<-}{R}{S}
\ncline[linewidth=1pt,linecolor=lightgray]{->}{S}{T}
\ncline[linewidth=1pt,linecolor=lightgray]{<-}{T}{A}
\end{pspicture}
\caption{Black arrows indicate divisors containing an \protect$E\protect$ that does not lie in
\protect$\supp(\Psi(S_i))\protect$}
\label{fig:vanishingType}
\end{figure}

\begin{proof}
 Let $E$ be a divisor that is not contained in the support of $\Psi(S_i)$. We first show how to obtain the indices $\mu, \nu$ for which $E$ labels only the divisors in the wheel $W_i$ as indicated.
  
 Suppose that $E$ is contained in none of the in-spoke divisors. For all $1\leq j\leq m$, Lemma~\ref{lem:technicalWheels}\one\ implies that $E$ is contained in either the out-spoke divisor $D_{j,j+1}$, or in $\gcd(D^j_{j+1},D_j^{j+1})$. If the latter occurs for two or more indices, say for $j,k$ with $1\leq k<j$, then Lemma~\ref{lem:H-1equivalent}\one\ implies that $E\subseteq \supp(H^{-1}(\Psi(S_i)))$ which is absurd. If the latter never occurs, then $E$ is contained in every out-spoke divisor $D_{1,2}, \dots, D_{m,1}$, so Proposition~\ref{prop:cohomology} gives $E\subseteq \supp(H^{-2}(\Psi(S_i)))$ and $i=0$ which is also absurd. Thus, there exists a unique index $1\leq \mu\leq m$ such that $E$ is contained in the out-spoke divisors $D_{1,2}, \dots, D_{\mu-1,\mu}, D_{\mu+1,\mu+2},\dots D_{m,1}$ and the rim divisors $D^\mu_{\mu+1}$ and $D_\mu^{\mu+1}$; this is the special case of the result with $\mu=\nu-1$.

Otherwise, $E$ is contained in at least one in-spoke divisor, say $D^{j}$ for $1\leq j\leq m$. Since $E$ does not lie in $\supp(H^0(\Psi(S_i)))$, Proposition~\ref{prop:cohomology} implies that some in-spoke divisor $D^{\lambda}$ does not contain $E$. Lemma~\ref{lem:technicalWheels}\two\ then produces indices $\mu, \nu$ with $\lambda< \mu+1\leq j\leq \nu-1$ such that $E$ is contained in the rim divisors $D^{\nu-1}_{\nu}, D^{\mu+1}_{\mu}$ and the in-spoke divisors $D^{\mu+1},\dots,D^{\nu-1}$, but $E$ is not contained in $D^{\mu}$ or $D^{\nu}$. In light of Lemma~\ref{lem:technicalWheels}\one, it remains to prove that $E$ is contained in the out-spoke divisors $D_{\nu,\nu+1},\dots, D_{\mu-1,\mu}$. We distinguish two cases:

\smallskip

\textsc{Case 1:} Suppose that the in-spoke divisor $D^1$ does not contain $E$, so we may set $\lambda=1$. Now, $E$ is not contained in any of the in-spoke divisors $D^{\nu+1},\dots, D^m$, otherwise Lemma~\ref{lem:H-1equivalent}\two\ with $\nu_k=\nu$ implies that $E\subseteq \supp(H^{-1}(\Psi(S_i)))$ which is absurd. Similarly, $E$ is not contained in $\gcd(D^{\alpha+1}_{\alpha}, D^{\alpha}_{\alpha+1})$ for all $\nu \leq \alpha \leq m$, otherwise Lemma~\ref{lem:H-1equivalent}\one\ with $k=\alpha$ implies that $E\subseteq \supp(H^{-1}(\Psi(S_i)))$ which is absurd. Lemma~\ref{lem:technicalWheels}\one\ then implies that $E$ is contained in the out-spoke divisors $D_{\nu,\nu+1}, \dots, D_{m,1}$. To show that $E$ is contained in $D_{1,2},\dots, D_{\mu-1,\mu}$, we argue by contradiction. Let $\alpha\in \{1,\dots, \mu - 1\}$ be the largest index such that $E$ is not contained in the out-spoke divisor $D_{\alpha-1,\alpha}$. Since $E$ is not contained in $D^{\mu}$, Lemma~\ref{lem:technicalWheels}\one\ implies that $E$ is contained in the rim divisor $D^{\alpha}_{\alpha-1}$. Moreover, $E$ is not contained in $D^{\alpha-1}_{\alpha}$, otherwise Lemma~\ref{lem:H-1equivalent}\one\ with $k=\alpha-1$ implies that $E\subseteq \supp(H^{-1}(\Psi(S_i)))$ which is absurd. Thus, $E$ is contained in the in-spoke divisor $D^{\alpha}$, but this too is absurd because Lemma~\ref{lem:H-1equivalent}\two\ with $\nu_k=\alpha+1$ implies that $E\subseteq \supp(H^{-1}(\Psi(S_i)))$. As a result, no such $\alpha$ exists after all, so $E$ is contained in the out-spoke divisors $D_{\nu,\nu+1},\dots, D_{\mu-1,\mu}$ as required.

\smallskip

\textsc{Case 2:} Otherwise, $D^1$ contains $E$. Then  $1\not\in \{\lambda, \mu,\nu\}$ because none of $D^\lambda$, $D^{\mu}$, $D^{\nu}$ contain $E$. Since  $\lambda < \mu+1\leq \nu-1$,  index 1 may lie in any of three possible intervals in the cyclic order:

\smallskip

\textsc{Subcase} (a): If $\lambda < 1 < \mu$,  Lemma~\ref{lem:technicalWheels}\two\ gives indices $\alpha,\beta$ such that all in-spoke divisors $D^{\alpha+1},\dots, D^1,\dots, D^{\beta-1}$ contain $E$, whereas $D^{\alpha}$ and $D^{\beta}$ do not. But then Lemma~\ref{lem:H-1equivalent}\three\ for the transposition $(\mu_k,\nu_k)=(\beta, \alpha)$ implies that $E\subseteq \supp(H^{-1}(\Psi(S_i)))$ which is absurd.

\smallskip

\textsc{Subcase} (b): If $\mu < 1 < \nu$, we claim that $E$ is contained in $D_{\nu,\nu+1},\dots, D_{\mu-1,\mu}$. Indeed, if $E$ is contained in $D^{\alpha}$ for any $\nu+1 \leq \alpha \leq \mu-1$, we may again apply Lemma~\ref{lem:H-1equivalent}\three\ to obtain $E\subseteq \supp(H^{-1}(\Psi(S_i)))$ which is absurd. Moreover, $E$ cannot be contained in $\gcd(D^{\alpha+1}_{\alpha}, D^{\alpha}_{\alpha+1})$ for any $\nu \leq \alpha \leq \mu-1$, otherwise Lemma~\ref{lem:H-1equivalent}\one\ implies that $E\subseteq \supp(H^{-1}(\Psi(S_i)))$ which is absurd. Lemma~\ref{lem:technicalWheels}\one\ then shows that $E$ is contained in $D_{\nu,\nu+1},\dots, D_{\mu-1,\mu}$ as claimed.

\smallskip

\textsc{Subcase} (c): If $\nu < 1 < \lambda$, apply Lemma~\ref{lem:technicalWheels}\two\ as in Subcase (a) to derive a contradiction.

\medskip

Thus, in either \textsc{Case 1} or \textsc{Case 2}, the fact that $E$ is not contained in $\supp(\Psi(S_i))$ forces the divisors in the wheel $W_i$ containing $E$ to be those indicated in Figure~\ref{fig:vanishingType}. In the course of the above, we see that the converse statement holds for any configuration as in Figure~\ref{fig:vanishingType}.
\end{proof}

\begin{proposition}
\label{prop:H-1divisor}
Let $i\in \qpol_0$ be a nonzero vertex. Every irreducible component in the support of $H^{-1}(\Psi(S_i))$ has dimension two.
\end{proposition}

\begin{proof}
 In light of Corollary~\ref{cor:dim1or2}, we need only show that an irreducible component in the support of $H^{-1}(\Psi(S_i))$ cannot be of the form $E_1\cap E_2$, where $E_1$ and $E_2$ are torus-invariant divisors. We suppose otherwise and seek a contradiction. Neither $E_1$ nor $E_2$ is contained in $\supp H^{-1}(\Psi(S_i))$, otherwise $E_1 \cap E_2$ wouldn't be an irreducible component. Moreover, $H^{k}(\Psi(S_i))=0$ for $k\neq -1$ by Theorem~\ref{thm:intromain1}\one, so neither $E_1$ nor $E_2$ is contained in the support of $\Psi(S_i)$. The support of $H^{-1}(\Psi(S_i))$ is the union of the subschemes $Z_j(i)$ for $1\leq j\leq n$ by Proposition~\ref{prop:H-1}, so the curve $E_1\cap E_2$ lies in $Z_j(i)$ for some $1\leq j\leq n$. We claim that either $E_1$ or $E_2$ labels every outspoke in $W_{i}$ and hence every out-spoke vanishes on the locus $E_1\cap E_2$. Assuming the claim, Proposition~\ref{prop:cohomology}\two\ implies that $E_1\cap E_2$ lies in $\supp H^{-2}(\Psi(S_i))$, but this contradicts the final clause in Proposition~\ref{prop:cohomology}\two\ since $i\neq 0$.  This proves the result assuming the claim.

 We prove the claim via case-by-case analysis of the value of $j$ according to Proposition~\ref{prop:H-1}.

\smallskip

\textsc{Case 1:} 
 For $1\leq j\leq m$, assume without loss of generality that $E_1$ lies in $\gcd(D_{j+1}^j,D_j^{j+1})$ and $E_2$ lies in $\lcm\big(D^1,\dots,D^m,\gcd(D_{j+2}^{j+1},D^{j+2}_{j+1}),\dots,\gcd(D_{1}^{m},D^{1}_{m})\big)-\lcm(D^j,D^{j+1})$. Since $E_1$ does not lie in $\supp(\Psi(S_i))$ and since $E_1$ lies in both $D_{j+1}^j$ and $D_j^{j+1}$ which label consecutive edges on the boundary of the wheel $W_{i}$, Proposition~\ref{prop:vanishing} implies that $E_1$ labels every out-spoke in $W_{i}$ except $D_{j,j+1}$. Now, since $E_2$ is not contained in $\supp(\Psi(S_i))$ we have that either:
\begin{enumerate}
\item $E_2$ is contained in $\gcd(D_{\mu+1}^\mu,D_\mu^{\mu+1})$ for some $j+1\leq \mu\leq m$. Proposition~\ref{prop:vanishing} then implies that $E_2$ labels every out-spoke in $W_{i}$ except $D_{\mu,\mu+1}$, so $E_2$ labels $D_{j,j+1}$; or
\item $E_2$ is contained in $D^\mu$ for some $1\leq \mu\leq m$. Since $E_2$ is contained in neither $D^j$ nor $D^{j+1}$, Lemma~\ref{lem:technicalWheels}\one\ implies that either $E_2$ labels $D_{j,j+1}$, or $E_2$ labels $D_{j+1}^j$ and $D_j^{j+1}$. In this latter case, Proposition~\ref{prop:vanishing} implies that $E_2$ labels no in-spokes, so it cannot label $D^\mu$ for some $1\leq \mu\leq m$, a contradiction. Thus, $E_2$ must label $D_{j,j+1}$ after all.
\end{enumerate}
 Thus, either $E_1$ or $E_2$ labels every outspoke in $W_{i}$ which proves the claim for $1\leq j\leq m$.

\smallskip

\textsc{Case 2:}
For $m+1 \leq j \leq 2m-3$, assume without loss of generality that $E_1$ is contained in $\lcm(D^1,D^{\nu_j-1},D^{\nu_j})-\lcm(D^1,D^{\nu_j})$ and $E_2$ is contained in $\lcm(D^1,D^{\nu_j},D^{\nu_j+1},\dots,D^m)-\lcm(D^1,D^{\nu_j})$. The condition on $E_1$ means that $E_1 \not\subseteq D^1$, $E\not\subseteq D^{\nu_j}$ and $E_1 \subseteq D^{\nu_j-1}$. Now, since $E_1$ is not contained in $\supp(\Psi(S_i))$, Proposition~\ref{prop:vanishing} implies that $E_1$ labels all of the out-spokes $D_{\nu_j,\nu_j+1}, D_{\nu_j+1,\nu_j+2},\dots, D_{m,1}$. The condition on $E_2$ means that $E_2 \not\subseteq D^1$ and $E_2\not\subseteq D^{\nu_j}$, while $E_2 \subseteq D^{\mu}$ for some $\nu_{j}+1 \leq \mu \leq m$. Since $E_2$ is not contained in $\supp(\Psi(S_i))$, Proposition~\ref{prop:vanishing} implies that $E_2$ labels all of the out-spokes $D_{1,2}, D_{2,3},\dots, D_{\nu_j-1,\nu_j}$. Thus, either $E_1$ or $E_2$ labels every outspoke in $W_{i}$ which proves the claim for $m+1 \leq j \leq 2m-3$.

\smallskip

\textsc{Case 3:}
Finally, for $2m-2 \leq j \leq n$, the curve $E_1 \cap E_2$ is contained in the intersection of the divisors $\lcm(D^{\mu},D^{\mu_j},D^{\nu_j})-\lcm(D^{\mu_j},D^{\nu_j})$ for $\mu \in \{1,\dots, \mu_j-1  \}\cup \{\nu_j-1\}$. At least one of $E_1, E_2$ is contained in the divisor with index $\nu_j-1$, and if both are then at least one is contained in a divisor with index $\mu \in \{1,\dots, \mu_j-1\}$. Thus, we may assume without loss of generality that $E_1$ is contained in a divisor with index $\mu \in \{1,\dots, \mu_j-1\}$ and $E_2$ is contained in the divisor with index $\nu_j-1$. This means firstly that $E_1\subseteq D^{\mu}$, $E_1\not\in D^{\mu_j}$ and $E_1\not\in D^{\nu_j}$. Since $E_1$ is not contained in $\supp(\Psi(S_i))$ and since $\mu, \mu_j, \nu_j$ are in cyclic order, Proposition~\ref{prop:vanishing} implies that $E_1$ labels all of the out-spokes $D_{\mu_j,\mu_j+1}, \dots, D_{\nu_j-1,\nu_j}$. Secondly, it means that $E_2$ labels the in-spoke  divisor $D^{\nu_j-1}$ and does not label the in-spoke divisors $D^{\mu_j}$ and $D^{\nu_j}$. Since $E_2$ does not lie in $\supp(\Psi(S_i))$ and since $\mu_j, \nu_j-1, \nu_j$ are in cyclic order, Proposition~\ref{prop:vanishing} implies that $E_2$ labels all of the out-spokes $D_{\nu_j,\nu_j+1}, \dots, D_{\mu_j-1,\mu_j}$. Thus, either $E_1$ or $E_2$ labels every outspoke in $W_{i}$ which proves the claim for $2m-3 \leq j \leq n$.
\end{proof}

\begin{proof}[Proof of Theorem~\ref{thm:intromain2}]
This follows from Propositions~\ref{prop:PsiS0}, \ref{prop:type0or1wall} and \ref{prop:H-1divisor}.
\end{proof}

\end{document}